\documentclass[12pt, reqno]{amsart}
\usepackage{xcolor} 
\usepackage{cancel}
\usepackage{stmaryrd}
\usepackage{amsmath}
\usepackage{amsfonts}
\usepackage{amssymb}
\usepackage[all,cmtip]{xy}           
\usepackage{xspace}
\usepackage{bbding}
\usepackage{txfonts}
\usepackage[shortlabels]{enumitem}
\usepackage{cite}
\usepackage{longtable}
\usepackage{makecell}
\usepackage{tikz}
\usepackage{titletoc}

\usepackage{ifpdf}
\ifpdf
 \usepackage[colorlinks,final,hyperindex]{hyperref}
\else
 \usepackage[colorlinks,final,backref=page,hyperindex,hypertex]{hyperref}
\fi
\usepackage[active]{srcltx}

\makeatletter

\begin{document}
\providecommand{\bluenote}[1]{\par\noindent{\footnotesize\color{blue}\textbf{[AUTHOR CHECK]} #1}\par\smallskip}

\newcommand {\emptycomment}[1]{}
\baselineskip=15pt
\newcommand{\nc}{\newcommand}
\newcommand{\delete}[1]{}
\newcommand{\problem}[1]{\textcolor{blue}{#1}}
\newcommand{\checknote}[1]{\par\noindent\textcolor{blue}{\textbf{[CHECK: }#1\textbf{]}}}
\nc{\mfootnote}[1]{\footnote{#1}}
\nc{\todo}[1]{\tred{To do:} #1}
\nc{\mlabel}[1]{\label{#1}}
\nc{\mcite}[1]{\cite{#1}}
\nc{\mref}[1]{\ref{#1}}
\nc{\meqref}[1]{\eqref{#1}}
\nc{\mbibitem}[1]{\bibitem{#1}}
\delete{
\nc{\mlabel}[1]{\label{#1}{\hfill \hspace{1cm}{\bf{{\ }\hfill(#1)}}}}
\nc{\mcite}[1]{\cite{#1}{{\bf{{\ }(#1)}}}}
\nc{\mref}[1]{\ref{#1}{{\bf{{\ }(#1)}}}}
\nc{\meqref}[1]{\eqref{#1}{{\bf{{\ }(#1)}}}}
\nc{\mbibitem}[1]{\bibitem[\bf #1]{#1}}
}
\newcommand {\comment}[1]{{\marginpar{*}\scriptsize\textbf{Comments:} #1}}
\nc{\mrm}[1]{{\rm #1}}
\nc{\id}{\mrm{id}}  \nc{\Id}{\mrm{Id}}
\nc{\admset}{\{\pm x\}\cup (-x+K^{\times}) \cup K^{\times} x^{-1}}
\def\a{\alpha}
\def\admt{admissible to~}
\def\ad{associative D-}
\def\asi{ASI~}
\def\aybe{aYBe~}
\def\b{\beta}
\def\bd{\boxdot}
\def\btl{\blacktriangleright}
\def\btr{\blacktriangleleft}
\def\calo{\mathcal{O}}
\def\ci{\circ}
\def\d{\delta}
\def\D{\Delta}
\newcommand{\End}{\mathrm{End}}
\def\frakB{\mathfrak{B}}
\def\G{\Gamma}
\def\g{\gamma}
\def\l{\lambda}
\def\lh{\leftharpoonup}
\def\lr{\longrightarrow}
\def\n{Nijenhuis~}
\def\o{\otimes}
\def\om{\omega}
\def\opa{\cdot_{A}}
\def\opb{\cdot_{B}}
\def\p{\psi}
\def\sadm{$S$-admissible~}
\def\r{\rho}
\def\ra{\rightarrow}
\def\rh{\rightharpoonup}
\def\rr{r^{\#}}
\def\s{\sigma}
\def\st{\star}
\def\ti{\times}
\def\tl{\triangleright}
\def\tr{\triangleleft}
\def\v{\varepsilon}
\def\vp{\varphi}
\def\vth{\vartheta}
\def\la{left-Alia~}
\def\rpn{representation~}
\newtheorem{thm}{Theorem}[section]
\newtheorem{lem}[thm]{Lemma}
\newtheorem{cor}[thm]{Corollary}
\newtheorem{pro}[thm]{Proposition}
\theoremstyle{definition}
\newtheorem{defi}[thm]{Definition}
\newtheorem{ex}[thm]{Example}
\newtheorem{rmk}[thm]{Remark}
\newtheorem{pdef}[thm]{Proposition-Definition}
\newtheorem{condition}[thm]{Condition}
\newtheorem{question}[thm]{Question}
\renewcommand{\labelenumi}{{\rm(\alph{enumi})}}
\renewcommand{\theenumi}{\alph{enumi}}
\nc{\ts}[1]{\textcolor{purple}{MTS:#1}}
\nc{\zc}[1]{\color{blue}{Zhao:#1}}
\font\cyr=wncyr10

\title{Bialgebras induced by special left-Alia algebras}
\author[Ma]{Tianshui Ma \textsuperscript{*}}
\address{School of Mathematics and Statistics, Henan Normal University, Xinxiang 453007, China}
\email{matianshui@htu.edu.cn}
\author[Zhao]{Chan Zhao}
\address{School of Mathematics and Statistics, Henan Normal University, Xinxiang 453007, China}
\email{zhaochan2024@stu.htu.edu.cn}
\author[Zheng]{Huihui Zheng}
\address{School of Mathematics and Statistics, Henan Normal University, Xinxiang 453007, China}
\email{zhenghuihui@htu.edu.cn}
\thanks{\textsuperscript{*}Corresponding author}
\date{\today}
\begin{abstract}
This paper investigates bialgebraic structures associated with special left-Alia algebras. Based on the special left-Alia algebras originally introduced in the study of algebras with skew-symmetric identity of degree 3, we systematically examine the structure of a Nijenhuis left-Alia bialgebra. This bialgebra consists of a compatible pairing between a Nijenhuis left-Alia algebra and a Nijenhuis left-Alia coalgebra. We prove that a Nijenhuis special left-Alia bialgebra, which is constructed from a commutative associative algebra (or a cocommutative coassociative coalgebra) equipped with two commuting linear maps, can be induced from a Nijenhuis commutative cocommutative associative D-bialgebra. This establishes an explicit link between these two families of bialgebraic structures. Furthermore, the paper provides a general method for constructing Nijenhuis operators on arbitrary left-Alia algebras and coalgebras.
\end{abstract}
\subjclass[2020]{17B38, 17A30, 16T25, 16T10.}
\keywords{special left-Alia algebra; Nijenhuis operator; Nijenhuis left-Alia bialgebra}
\maketitle
\allowdisplaybreaks

\section{Preliminaries and Introduction}
Nijenhuis operators are an important class of operators in mathematics, with widespread applications in geometry, algebra, and analysis. Serving as a bridge connecting geometry, algebra, analysis, and discrete mathematics, the study of the Nijenhuis operator continuously drives the development of multiple disciplines.
Let $(A, \diamond)$ be an algebra. Then a {\bf Nijenhuis operator on $(A, \diamond)$} is a linear map $N:A\lr A$ such that for all $x, y\in A$, the following identity holds: 
\begin{eqnarray}
N(x)\diamond N(y)+N^2(x\diamond y)=N(N(x)\diamond y+x\diamond N(y)).\label{eq:n}\label{eq:7}
\end{eqnarray}
In this case, we call the pair $((A, \diamond), N)$ a {\bf Nijenhuis algebra}.

A {\bf left anti-Lie-admissible algebra (abbr. left-Alia algebra)} (also called a 0-Alia algebra) is a pair $(A,[,])$ consisting of a vector space $A$ and a bilinear map $[,]: A\o A\rightarrow A$ satisfying the symmetric Jacobi identity
\begin{eqnarray}\label{eq:syjacobi}
[[x, y], z]+[[y, z], x]+[[z, x], y]=[[y, x], z]+[[z, y], x]+[[x, z], y],
\end{eqnarray}
where $x,y,z\in A$. This class of algebras was introduced by Dzhumadil'daev in \cite{Dz} when studying the classification of algebras with skew-symmetric identity of degree 3. There are various typical examples of left-Alia algebras, including Lie algebras, anti-pre-Lie algebras and mock-Lie algebras. 

Let $(A, \cdot)$ be a commutative associative algebra, and let $f, g: A\lr A$ be linear maps. Define a multiplication $[,]_{(f,g)}$ on $A$ by
\begin{eqnarray}
[x, y]_{(f,g)}:=x\cdot f(y)+g(x\cdot y),~\forall~x, y\in A.\label{eq:special}
\end{eqnarray}
Then by \cite[Theorem 6.2]{Dz}, we know that $(A, [,]_{(f,g)})$ is a left-Alia algebra, which is called by Dzhumadil'daev a {\bf special left-Alia algebra with respect to $(A, \cdot, f, g)$}. Furthermore, we find that if $f$ is a Nijenhuis operator on $(A, \cdot)$ and $f\ci g=g\ci f$, then $f$ is a Nijenhuis operator on $(A, [,]_{(f,g)})$ (see Theorem \ref{thm:q}). Dually, let $(A, \d)$ be a cocommutative coassociative coalgebra, $F, G: A\lr A$ be linear maps. Define a comultiplication $\D_{(F,G)}$ on $A$ by
\begin{eqnarray}\label{eq:cosplaa}
\D_{(F, G)}(x)=x_{[1]}\o F(x_{[2]})+G(x)_{[1]}\o G(x)_{[2]},
\end{eqnarray}
where $\d(x)=x_{[1]}\o x_{[2]},~ \forall~x\in A$. Then $(A, \D_{(F,G)})$ is a left-Alia coalgebra, which is called a {\bf special left-Alia coalgebra with respect to $(A, \d, F, G)$}. Further, if $F$ is a Nijenhuis operator on $(A, \d)$ and $F\ci G=G\ci F$, then $F$ is a Nijenhuis operator on $(A, \D_{(F,G)})$ (see Proposition \ref{pro:ahhh}).

In 2024, a bialgebraic theory for Nijenhuis associative algebras was presented and the notion of a Nijenhuis associative D-bialgebra $((A, \cdot, \d), f, F)$ was introduced by combining a Nijenhuis associative algebra $((A, \cdot), f)$ and a Nijenhuis coassociative coalgebra $((A, \d), F)$ \cite{MLo}. Is there a bialgebraic structure $((A, [,]_{(f, g)}, \D_{(F, G)}), f, F)$ on the special left-Alia algebra $(A, [,]_{(f,g)})$ that can be derived by a Nijenhuis commutative and cocommutative associative D-bialgebra $((A, \cdot, \d), f, F)$? That is to say,
{\small \begin{center}\vskip28mm
\unitlength 1mm
\linethickness{0.4pt}
\ifx\plotpoint\undefined\newsavebox{\plotpoint}\fi
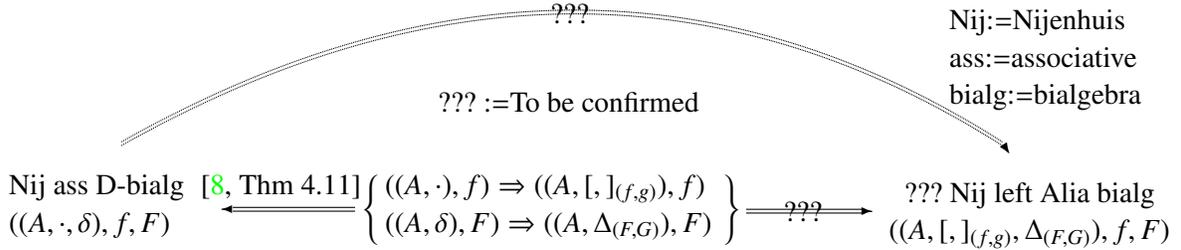
\hspace{-54mm}
\begin{picture}(105,47.25)(0,0)
\put(3,48){Nij ass D-bialg}
\put(3,43){$((A, \cdot, \d), f, F)$}
\put(50,45){$\Bigg\{$}
\put(53,48){$((A, \cdot), f)\Rightarrow ((A, [,]_{(f, g)}), f)$}
\put(53,43){$((A, \d), F)\Rightarrow ((A, \D_{(F, G)}), F)$}
\put(98,45){$\Bigg\}$}
\put(122,47){?~Nij \la bialg}
\put(120,42){$((A, [,]_{(f, g)}, \D_{(F, G)}), f, F)$}
\put(32,46){\vector(-1,0){0.75}}
\put(49,46.25){\line(-1,0){16.25}}
\put(49,45.75){\line(-1,0){16.25}}
\put(117,45.7){\vector(1,0){0.75}}
\put(101,45.5){\line(1,0){15.75}}
\put(101,46){\line(1,0){15.75}}
\put(74,70.8){ ? }
\put(105,44.5){ ? }
\put(128,70){Nij:=Nijenhuis}
\put(128,65){ass:=associative}
\put(128,60){bialg:=bialgebra}
\put(60,59){?~:=To be confirmed}
\put(27.5,48){\small\cite[Thm 4.11]{MLo}}
\put(135.25,54.5){\vector(2,-3){0.75}}\qbezier(18,55)(80,90)(135,55)
\put(135.25,54.5){\vector(2,-3){0.75}}\qbezier(18,54.5)(80,89.5)(135,54.5)
\end{picture}
\end{center}}
\vskip-39mm
This is the motivation for writing this article. 

In this paper we will explore the bialgebraic structure on a Nijenhuis left-Alia algebra such that the special case $((A, [,]_{(f,g)}, \D_{(g,f)}), f, g)$ can be induced by Nijenhuis commutative cocommutative associative D-bialgebra $((A, \cdot, \d), f, g)$. Also we find that Nijenhuis operators can be derived from a class of left-Alia bialgebras. 

This paper is organized as follows: In Section \ref{se:nijbialg}, we give the bialgebraic structure on a Nijenhuis left-Alia algebra by means of the equivalence between matched pairs of Nijenhuis left-Alia algebras and Manin triples of Nijenhuis left-Alia algebras. In Section \ref{se:ybe}, we focus on a class of Nijenhuis left-Alia bialgebras, named triangular Nijenhuis left-Alia bialgebras, which yields a compatibility condition between the left-Alia Yang-Baxter equation and Nijenhuis operators. We also provide some constructions of Nijenhuis left-Alia bialgebras through relative Rota-Baxter operators. In Section \ref{se:special}, we apply the bialgebraic theory to Nijenhuis special left-Alia algebras and obtain a necessary and sufficient condition for a Nijenhuis special left-Alia algebra and a Nijenhuis special left-Alia coalgebra to form a Nijenhuis special left-Alia bialgebra. We also find that a class of Nijenhuis special left-Alia algebras and Nijenhuis special left-Alia coalgebras automatically produce Nijenhuis special left-Alia bialgebras, thereby realizing the motivation of this article. In Section \ref{se:nijcon}, we present a method to construct Nijenhuis operators on left-Alia (co)algebras based on a class of left-Alia bialgebras, illustrating the close connection between left-Alia algebras and their Nijenhuis operators. 

{\bf Notations:} Throughout this paper, all vector spaces, tensor products and linear homomorphisms are over a field $K$ of characteristic zero. We denote by $\id_M$ the identity map from $M$ to $M$ and by $\tau: M\o M\lr M\o M$ the flip map. We freely use the Sweedler notation for coalgebras in \cite{Sw}. For brevity, if $(C, \D)$ is a coalgebra, we write the comultiplication $\D(c)$ as $c_{(1)}\o c_{(2)}, ~\forall~c\in C$ without the summation sign.

\section{Nijenhuis left-Alia bialgebras}\label{se:nijbialg} In this section, we investigate the bialgebraic structures on a Nijenhuis left-Alia algebra.

\subsection{Nijenhuis left-Alia algebras} Let us recall from \cite{KLWY} some notions and results about left-Alia algebras. 
\begin{defi}\label{de:b}\label{de:cc} Let $(A, [, ])$ be a left-Alia algebra. A {\bf representation} of $(A, [, ])$ is a triple $(V, \ell, r)$, where $V$ is a vector space and $\ell, r: A\lr \End(V)$ are linear maps such that the following equation holds:
\begin{eqnarray}
\ell([x, y])-\ell([y, x])=r(x)(r(y)-\ell(y))-r(y)(r(x)-\ell(x)), \forall~ x, y\in A.\label{eq:2}
\end{eqnarray}
Let $(V_{A}, \ell_{A}, r_{A})$ and $(V_{B}, \ell_{B}, r_{B})$ be representations of $(A, [,]_{A})$ and $(B, [,]_{B})$ respectively. Then $(\phi, f)$ is a {\bf homomorphism from $(V_{A}, \ell_{A}, r_{A})$ of $(A, [,]_{A})$ to $(V_{B}, \ell_{B}, r_{B})$ of $(B, [,]_{B})$} if there are linear maps $\phi: V_{A}\lr V_{B}$ and $f: A\lr B$ such that $f([x, y]_{A})=[f(x), f(y)]_{B}$ and
\begin{eqnarray}
&\phi(\ell_{A}(x)(v))=\ell_{B}(f(x))(\phi(v)), ~~~ \phi(r_{A}(x)(v))=r_{B}(f(x))(\phi(v)), &\label{eq:cc1}
\end{eqnarray}
where $x, y\in A, v\in V_A$.
\end{defi}

\begin{pro}\label{pro:c} Let $(A, [, ])$ be a left-Alia algebra, $V$ be a vector space and $\ell, r: A\lr \End(V)$ be linear maps. Then $(V, \ell, r)$ is a representation of $(A, [, ])$ if and only if there is a left-Alia algebra $(A\oplus V, [, ]_{A\oplus V})$ on the direct sum $A\oplus V$ of vector spaces, where $[, ]_{A\oplus V}$ is given by
\begin{eqnarray}
&[x+u, y+v]_{A\oplus V}:=[x, y]+\ell(x)v+r(y)u, \forall~ x, y\in A, u, v\in V. &\label{eq:3}
\end{eqnarray}
We denote this left-Alia algebra by $A\ltimes_{\ell, r} V$ and call it {\bf semi-direct product of $(A, [,])$ by $(V, \ell, r)$}.
\end{pro}

\begin{ex}\label{ex:d} Let $(A, [, ])$ be a left-Alia algebra. Define linear maps $\mathcal{L}, \mathcal{R}: A\lr \End(A)$ by
\begin{eqnarray}
&\mathcal{L}(x)y=[x, y]=\mathcal{R}(y)x, \forall~ x, y \in A. &\label{eq:4}
\end{eqnarray}
Then $(A, \mathcal{L}, \mathcal{R})$ is a representation of $(A, [, ])$, which is called an {\bf adjoint representation}.
\end{ex}

In what follows, we introduce the notion of Nijenhuis left-Alia algebras. 
\begin{lem} \label{lem:ff} Let $[, ], [, ]_{\cdot}: A\o A\lr A$, and $\ell_{1}, \ell_{2}, r_{1}, r_{2}: A\lr \End(V)$ be linear maps, $s, t$ be parameters. Consider the following new linear maps:
$$
[x, y]_{\circ}=s [x, y]+t [x, y]_{\cdot}, \ \  \ell(x)=s \ell_{1}(x)+t \ell_{2}(x), \ \ r(x)=s r_{1}(x)+t r_{2}(x), \forall~ x, y\in A.
$$
Then
\begin{enumerate}[(1)]
\item \label{it:1} $(A, [, ]_{\circ})$ is a left-Alia algebra if and only if $(A, [, ])$ and $(A, [, ]_{\cdot})$ are left-Alia algebras, and for all $x, y, z \in A$,
\begin{eqnarray*}
&&[[x, y]_{\cdot}-[y, x]_{\cdot}, z]+[[y, z]_{\cdot}-[z, y]_{\cdot}, x]+[[z, x]_{\cdot}-[x, z]_{\cdot}, y]\\
&&+[[x, y]-[y, x], z]_{\cdot}+[[y, z]-[z, y], x]_{\cdot}+[[z, x]-[x, z], y]_{\cdot}=0.
\end{eqnarray*}
\item \label{it:2} $(V, \ell, r)$ is a representation of $(A, [, ]_{\circ})$ if and only if $(V, \ell_{1}, r_{1})$ is a representation of $(A, [, ])$, $(V, \ell_{2}, r_{2})$ is a representation of $(A, [, ]_{\cdot})$, and for all $x, y \in A$,
\begin{eqnarray*}
&&\ell_{1}([x, y]_{\cdot}-[y, x]_{\cdot})+\ell_{2}([x, y]-[y, x])=r_{2}(x)(r_{1}(y)-\ell_{1}(y))\\
&&\qquad+r_{1}(x)(r_{2}(y)-\ell_{2}(y))-r_{2}(y)(r_{1}(x)-\ell_{1}(x))
-r_{1}(y)(r_{2}(x)-\ell_{2}(x)).
\end{eqnarray*}
\end{enumerate}
\end{lem}
\begin{proof} Straightforward.\end{proof}

\begin{thm}\label{thm:fff} Assume that $(A, [, ]_{\circ})$, $\ell$ and $r$ are defined in Lemma \ref{lem:ff}, $(V, \ell, r)$ is a representation of $(A, [, ]_{\circ})$, $(V, \ell_{1}, r_{1})$ is a representation of $(A, [, ])$. Let $N\in \End(A)$ and $\a\in \End(V)$ be two linear maps. Then $(s \id_{A}+tN, s \id_{V}+t\a)$ is a homomorphism from $(V, \ell, r)$ of $(A, [, ]_{\circ})$ to $(V, \ell_{1}, r_{1})$ of $(A, [, ])$ if and only if, for all $x, y\in A$ and $v\in V$,
\begin{eqnarray}
&[x, y]_{\cdot}=[N(x), y]+[x, N(y)]-N([x, y]),&\label{eq:nn1}\\
&N([x, y]_{\cdot})=[N(x), N(y)],&\label{eq:nn2}\\
&\ell_{2}(x)(v)=\ell_{1}(N(x))(v)+\ell_{1}(x)(\a(v))-\a(\ell_{1}(x)(v)),&\label{eq:nn3}\\
&\a(\ell_{2}(x)(v))=\ell_{1}(N(x))(\a(v)),&\label{eq:nn4}\\
&r_{2}(x)(v)=r_{1}(N(x))(v)+r_{1}(x)(\a(v))-\a(r_{1}(x)(v)),&\label{eq:nn5}\\
&\a(r_{2}(x)(v))=r_{1}(N(x))(\a(v)).&\label{eq:nn6}
\end{eqnarray}
\end{thm}
\begin{proof} For any $x, y\in A$, $v\in V$, according to Definition \ref{de:b}, we calculate as follows:
\begin{eqnarray*}
&&(s \id_{A}+t N)([x, y]_{\circ})-[(s \id_{A}+t N)(x), (s \id_{A}+t N)(y)]\\
&&\quad=(s \id_{A}+t N)(s[x, y]+t[x, y]_{\cdot})-[(s x+t N(x)), (s y+t N(y))]\\
&&\quad=s^{2}[x, y]+s t[x, y]_{\cdot}+ t s N([x, y])+t^{2}N([x, y]_{\cdot})-s^{2}[x, y]-s t[x, N(y)]-t s[N(x), y]\\
&&\qquad-t^{2}[N(x), N(y)]\\
&&\quad=s t([x, y]_{\cdot}+N([x, y])-[x, N(y)]-[N(x), y])+t^{2}(N([x, y]_{\cdot})-[N(x), N(y)])=0\\
&&(s \id_{V}+t\a)(\ell(x)(v))-\ell_{1}(s x+t N(x))(s \id_{V}+t\a)(v)\\
&&\quad=(s \id_{V}+t\a)(s \ell_{1}(x)(v)+t \ell_{2}(x)(v))-(s \ell_{1}(x)+t \ell_{1}(N(x)))(s \id_{V}+t\a)(v)\\
&&\quad=s^{2}\ell_{1}(x)(v)+s t \ell_{2}(x)(v)+t s\a(\ell_{1}(x)(v))+t^{2}\a(\ell_{2}(x)(v))-s^{2}\ell_{1}(x)(v)-t s \ell_{1}(N(x))(v)\\
&&\qquad-s t \ell_{1}(x)\a(v)-t^{2}\ell_{1}(N(x))\a(v)\\
&&\quad=s t(\ell_{2}(x)(v)+\a(\ell_{1}(x)(v))-\ell_{1}(N(x))(v)-\ell_{1}(x)\a(v))
+t^{2}(\a(\ell_{2}(x)(v))-\ell_{1}(N(x))\a(v))\\
&&\quad=0.
\end{eqnarray*}
Similarly,
\begin{eqnarray*}
&&(s \id_{V}+t\a)(r(x)(v))-r_{1}(s x+t N(x))(s \id_{V}+t\a)(v)\\
&&\quad=s t(r_{2}(x)(v)+\a(r_{1}(x)v)-r_{1}(N(x))v-r_{1}(x)\a(v))
+t^{2}(\a(r_{2}(x)(v))-r_{1}(N(x))\a(v))\\
&&\quad=0.
\end{eqnarray*}
Thus we finish the proof.\end{proof}

By Eqs.\eqref{eq:nn1} and \eqref{eq:nn2} we have:
\begin{defi}\label{de:g} A {\bf Nijenhuis left-Alia algebra} is a pair $((A, [, ]), N)$, where $(A, [, ])$ is a left-Alia algebra and $N$ is a Nijenhuis operator, i.e., Eq.(\ref{eq:7}) holds for $[,]$.
\end{defi}
By Eqs.\eqref{eq:nn3}-\eqref{eq:nn6} we have:
\begin{defi}\label{de:l} Let $((A, [, ]), N)$ be a Nijenhuis left-Alia algebra. A {\bf representation} of $((A, [, ]), N)$ is a pair $((V, \ell, r), \a)$, where $(V, \ell, r)$ is a representation of $(A, [, ])$ and $\a: V\lr V$ is a linear map such that, for all $x\in A, v\in V$, the following equations hold:
\begin{eqnarray}
&\ell(N(x))(\a(v))+\a^{2}(\ell(x)(v))=\a(\ell(N(x))(v))+\a(\ell(x)(\a(v))),\label{eq:8}\\
&r(N(x))(\a(v))+\a^{2}(r(x)(v))=\a(r(N(x))(v))+\a(r(x)(\a(v))).\label{eq:9}
\end{eqnarray}
Let $((V_{1}, \ell_{1}, r_{1}), \a_{1})$ and $((V_{2}, \ell_{2}, r_{2}), \a_{2})$ be representations of $((A, [, ]), N)$. A linear map $\phi: V_{1}\lr V_{2}$ is called a {\bf homomorphism from $((V_{1}, \ell_{1}, r_{1}), \a_{1})$ to $((V_{2}, \ell_{2}, r_{2}), \a_{2})$} if $(\phi, \id_A)$ is a homomorphism from $(V_{1}, \ell_{1}, r_{1})$ to $(V_{2}, \ell_{2}, r_{2})$ and, for all $v\in V_{1}$,
\begin{eqnarray}\label{eq:10}
\phi(\a_{1}(v))=\a_{2}(\phi(v)).
\end{eqnarray}
If $\phi$ is bijective, then we say that the representations $((V_{1}, \ell_{1}, r_{1}), \a_{1})$ and $((V_{2}, \ell_{2}, r_{2}), \a_{2})$ are {\bf equivalent}.
\end{defi}
\begin{ex}\label{ex:m}
\begin{enumerate}[(1)]
\item Let $((A, [, ]), N)$ be a Nijenhuis left-Alia algebra. Then $((A, \mathcal{L}, \mathcal{R}), N)$ is a representation of $((A, [, ]), N)$, called the {\bf adjoint representation} of $((A, [, ]), N)$.
\item Let $((V, \rho), \a)$ be a representation of a Nijenhuis Lie algebra $((A, [, ]), N)$, then, both $((V, \rho, -\rho), \a)$ and $((V, \rho, 2\rho), \a)$ are representations of $((A, [, ]), N)$ as a Nijenhuis left-Alia algebra.
\end{enumerate}
\end{ex}
\begin{pro}\label{pro:n} Let $((V, \ell, r), \a)$ be a representation of a Nijenhuis left-Alia algebra $((A, [, ]), N)$. Define the bracket product on $A\oplus V$ by Eq. $\eqref{eq:3}$, and a linear map $N+\a: A\oplus V\lr A\oplus V$ by
\begin{eqnarray}
(N+\a)(x+v):=N(x)+\a(v), \forall~ x\in A, \forall~ v\in V.\label{eq:11}
\end{eqnarray}
Then $((A\oplus V, [, ]_{A\oplus V}), N+\a)$ is a Nijenhuis left-Alia algebra denoted by $(A\ltimes_{\ell, r}V, N+\a)$ and called {\bf semi-direct product of $((A, [,]), N)$ by $((V, \ell, r), \a)$}.
\end{pro}
\begin{proof}
It is a special case of the matched pair of Nijenhuis left-Alia algebras in Theorem \ref{thm:an}, so we omit the proof here.
\end{proof}

\subsection{Dual representation}
Let $A$ and $V$ be vector spaces. For a linear map $\rho: A\lr \End(V)$, define a linear map $\rho^{*}: A\lr \End(V^{*})$ by
\begin{eqnarray}
&\langle \rho^{*}(x)(u^{*}), v\rangle=-\langle u^{*}, \rho(x)(v)\rangle, \forall~ x\in A, u^{*}\in V^{*}, v\in V. &\label{eq:repdual}
\end{eqnarray}
By \cite{KLWY}, if $(V, \ell, r)$ is a representation of $(A, [, ])$, then $(V^{*}, \ell^{*}, \ell^{*}-r^{*})$ is a representation of $(A, [,])$ naturally. For Nijenhuis left-Alia algebras, however, we have 
\begin{lem}\label{lem:o} Let $((A, [, ]), N)$ be a Nijenhuis left-Alia algebra, $(V, \ell, r)$ a representation of $(A, [, ])$ and $\b: V\lr V$ a linear map. Then $((V^{*}, \ell^{*}, \ell^{*}-r^{*}), \b^{*})$ is a representation of $((A, [, ]), N)$ if and only if, for all $x\in A$, $v\in V$,
\begin{eqnarray}
&\b(\ell(N(x))(v))+\ell(x)(\b^{2}(v))=\ell(N(x))(\b(v))+\b(\ell(x)(\b(v))),&\label{eq:13}\\
&\b(r(N(x))(v))+r(x)(\b^{2}(v))=r(N(x))(\b(v))+\b(r(x)(\b(v))).\label{eq:14}
\end{eqnarray}
In particular, for a linear map $S: A\lr A$, $((A^{*}, \mathcal{L}^{*}, \mathcal{L}^{*}-\mathcal{R}^{*}), S^{*})$ is a representation of $((A, [, ]), N)$ if and only if, for all $x,y\in A$,
\begin{eqnarray}
&S[N(x), y]+[x, S^{2}(y)]=S[x, S(y)]+[N(x), S(y)],&\label{eq:15}\\
&S[x, N(y)]+[S^{2}(x), y]=S[S(x), y]+[S(x), N(y)].&\label{eq:16}
\end{eqnarray}
\end{lem}
\begin{proof} By \cite{KLWY}, we know that $(V^{*}, \ell^{*}, \ell^{*}-r^{*})$ is a representation of $(A, [, ])$. And for all $x\in A$, $v\in V$, $u^{*}\in V^{*}$,
\begin{eqnarray*}
&&\big\langle \ell^{*}(N(x))(\b^{*}(u^{*}))+\b^{*2}(\ell^{*}(x)(u^{*}))-\b^*(\ell^{*}(N(x))(u^{*}))
-\b^*(\ell^{*}(x)(\b^{*}(u^{*}))), v \big\rangle\\
&&\qquad\stackrel{\eqref{eq:repdual}}{=}\big\langle u^{*}, -\b(\ell(N(x))(v))-\ell(x)(\b^{2}(v))+\ell(N(x))(\b(v))+\b(\ell(x)(\b(v)))\big\rangle.
\end{eqnarray*}
So Eq.(\ref{eq:8}) holds for $\ell^*$ and $\b^*$ if and only if Eq.(\ref{eq:13}) holds. Likewise, Eq.(\ref{eq:9}) holds for $\ell^*-r^*$ and $\b^*$ if and only if Eq.(\ref{eq:14}) holds.
\end{proof}
\begin{defi}\label{de:p} With notations in Lemma \ref{lem:o}, if Eqs.\eqref{eq:13} and \eqref{eq:14} hold, then we say that $\b$ is {\bf admissible to $(A, [,], N)$ with respect to $(V, \ell, r)$}. If Eqs.\eqref{eq:15} and \eqref{eq:16} hold, then we say that $S$ is {\bf adjoint-admissible to $((A, [, ]), N)$} or {\bf $((A, [, ]), N)$ is $S$-adjoint-admissible}.
\end{defi}

\subsection{Matched pairs of Nijenhuis left-Alia algebras}
Next we give the notion of a matched pair of Nijenhuis left-Alia algebras which extends the one of left-Alia algebras.
\begin{pro}\label{pro:al}{\em \cite{KLWY}} Let $(A, [, ]_{A})$ and $(B, [, ]_{B})$ be left-Alia algebras, $\ell_{A}, r_{A}: A\lr \End(B)$ and $\ell_{B}, r_{B}: B\lr \End(A)$ be linear maps. Define the bracket $[, ]_{A\oplus B}$ on the direct sum $A\oplus B$ of vector spaces given by
\begin{eqnarray}\label{eq:sum}
[x+a, y+b]_{A\oplus B}:=[x, y]_{A}+\ell_{B}(a)(y)+r_{B}(b)(x)+[a, b]_{B}+\ell_{A}(x)(b)+r_{A}(y)(a),
\end{eqnarray}
where $x, y\in A$ and $a, b\in B$, then $(A\oplus B, [, ]_{A\oplus B})$ is a left-Alia algebra if and only if $((A, [, ]_{A}), (B, [, ]_{B}), \ell_{A}, r_{A}, \ell_{B}, r_{B})$ is a matched pair of left-Alia algebras, that is, $(B, \ell_{A}, r_{A})$ is a representation of $(A, [, ]_{A})$, $(A, \ell_{B}, r_{B})$ is a representation of $(B, [, ]_{B})$, and the following equations hold:
\begin{eqnarray*}
r_{B}(a)([x, y]_{A}-[y, x]_{A})
&=&[(\ell_{B}-r_{B})(a)(y), x]_{A}+[(r_{B}-\ell_{B})(a)(x), y]_{A}\\
&&+\ell_{B}((r_{A}-\ell_{A})(y)(a))(x)+\ell_{B}((\ell_{A}-r_{A})(x)(a))(y),\\
r_{A}(x)([a, b]_{B}-[b, a]_{B})
&=&[(\ell_{A}-r_{A})(x)(b), a]_{B}+[(r_{A}-\ell_{A})(x)(a), b]_{B}\\
&&+\ell_{A}((r_{B}-\ell_{B})(b)(x))(a)+\ell_{A}((\ell_{B}-r_{B})(a)(x))(b).
\end{eqnarray*}
\end{pro}
\begin{defi}\label{de:am} A {\bf matched pair of Nijenhuis left-Alia algebras} $((A, [, ]_{A}), N_{A})$ and $((B,$ $[, ]_{B}), N_{B})$ is a six-tuple $((A, N_{A}), (B, N_{B}), \ell_{A}, r_{A}, \ell_{B}, r_{B})$, where $((B, \ell_{A}, r_{A}), N_{B})$ is a representation of $((A, [, ]_{A}), N_{A})$, $((A, \ell_{B}, r_{B}), N_{A})$ is a representation of $((B, [, ]_{B}), N_{B})$, and $((A, [, ]_{A}), (B,$ $[, ]_{B}), \ell_{A}, r_{A}, \ell_{B}, r_{B})$ is a matched pair of left-Alia algebras $(A, [, ]_{A})$ and $(B, [, ]_{B})$.
\end{defi}
\begin{thm}\label{thm:an} Let $((A, [, ]_{A}), N_{A})$ and $((B, [, ]_{B}), N_{B})$ be Nijenhuis left-Alia algebras. Then $((A, N_{A}), (B, N_{B}), \ell_{A}, r_{A}, \ell_{B}, r_{B})$ is a matched pair of $((A, [, ]_{A}), N_{A})$ and $((B, [, ]_{B}), N_{B})$ if and only if $((A\oplus B, [,]_{A\oplus B}), N_{A}+N_{B})$ is a Nijenhuis left-Alia algebra by defining the multiplication on $A\oplus B$ by Eq.\eqref{eq:sum} and linear map $N_{A}+N_{B}: A\oplus B\lr A\oplus B$ by
\begin{eqnarray}
(N_{A}+N_{B})(x+a):=N_{A}(x)+N_{B}(a)\label{eq:avnj}
\end{eqnarray}
for all $x\in A$, $a\in B$.
\end{thm}
\begin{proof} It can be proved by Proposition \ref{pro:al} and the following equality
{\small\begin{eqnarray*}
&&[(N_{A}+N_{B})(x+a), (N_{A}+N_{B})(y+b)]_{A\oplus B}+(N_{A}+N_{B})^{2}([(x+a), (y+b)]_{A\oplus B})\\
&&- (N_{A}+N_{B})([(N_{A}+N_{B})(x+a), (y+b)]_{A\oplus B})-(N_{A}+N_{B})([(x+a), (N_{A}+N_{B})(y+b)]_{A\oplus B})\\
&=&[N_{A}(x), N_{A}(y)]_{A}+N_{A}^{2}([x, y]_{A})-N_{A}([x, N_{A}(y)]_{A})-N_{A}([N_{A}(x), y]_{A})+[N_{B}(a), N_{B}(b)]_{B}\\
&&+N_{B}^{2}([a,b]_{B})-N_{B}([a, N_{B}(b)]_{B})-N_{B}([N_{B}(a), b]_{B})-N_{B}(\ell_{A}(x)(N_{B}(b)))+\ell_{A}(N_{A}(x))(N_{B}(b))\\
&&+N_{B}^{2}(\ell_{A}(x)(b))-N_{B}(\ell_{A}(N_{A}(x))(b))+r_{A}(N_{A}(y))(N_{B}(a))
+N_{B}^{2}(r_{A}(y)(a))-N_{B}(r_{A}(y)(N_{B}(a)))\\
&&-N_{B}(r_{A}(N_{A}(y))(a))-N_{A}(\ell_{B}(N_{B}(a))(y))+\ell_{B}(N_{B}(a))(N_{A}(y))
+N_{A}^{2}(\ell_{B}(a)(y))-N_{A}(\ell_{B}(a)(N_{A}(y)))\\
&&+r_{B}(N_{B}(b))(N_{A}(x))+N_{A}^{2}(r_{B}(b)(x))-N_{A}(r_{B}(b)(N_{A}(x)))
-N_{A}(r_{B}(N_{B}(b))(x)).
\end{eqnarray*}}
where $x, y\in A, a, b\in B$.
\end{proof}

\subsection{Manin triple of Nijenhuis left-Alia algebras}
Recall from \cite{KLWY} that a {\bf (symmetric) quadratic left-Alia algebra} is a triple $(A, [, ], \mathcal{B})$, where $(A, [,])$ is a left-Alia algebra and $\mathcal{B}$ is a nondegenerate (symmetric) bilinear form on $A$ which is invariant in the sense that
\begin{eqnarray}\label{eq:quadratic}
\mathcal{B}([x, y], z)=\mathcal{B}(x, [z, y]-[y, z]), \forall~ x, y, z\in A.
\end{eqnarray}
Since $\mathcal{B}$ is symmetric,
\begin{eqnarray}\label{eq:invariant}
\mathcal{B}([x, y], z)=-\mathcal{B}(y, [x, z]),\forall~ x, y, z\in A.
\end{eqnarray}

\begin{pro}\label{pro:ao} Let $((A, [, ]), N)$ be a Nijenhuis left-Alia algebra and $(A, [, ], \mathcal{B})$ a symmetric quadratic left-Alia algebra. Assume that $\widehat{N}$ is the adjoint linear map of $N$ with respect to $\mathcal{B}$, characterized by
\begin{eqnarray}\label{adjoint B}
\mathcal{B}(N(x), y)=\mathcal{B}(x, \widehat{N}(y)), \forall~ x, y\in A.
\end{eqnarray}
Then $\widehat{N}$ is adjoint-\admt $((A, [, ]), N)$, or equivalently, $((A^{*}, \mathcal{L}^{*}, \mathcal{L}^{*}-\mathcal{R}^{*}), \widehat{N}^{*})$ is a representation of $((A, [, ]), N)$. Moreover, $((A^{*}, \mathcal{L}^{*}, \mathcal{L}^{*}-\mathcal{R}^{*}), \widehat{N}^{*})$ is equivalent to $((A, \mathcal{L}, \mathcal{R}), N)$ as representations of $((A, [, ]), N)$. Conversely, let $((A, [, ]), N)$ be a Nijenhuis left-Alia algebra and $S: A\lr A$ a linear map that is adjoint-\admt $((A, [,]), N)$. If the resulting representation $((A^{*}, \mathcal{L}^{*}, \mathcal{L}^{*}-\mathcal{R}^{*}), S^{*})$ of $((A, [, ]), N)$ is equivalent to $((A, \mathcal{L}, \mathcal{R}), N)$, then there exists a quadratic left-Alia algebra $(A, [, ], \mathcal{B})$ such that $S=\widehat{N}$.
\end{pro}
\begin{proof} For all $x, y, z\in A$, we have {\small
\begin{eqnarray*}
0&=&\mathcal{B}([N(x), N(y)]+N^{2}([x, y])-N([N(x), y])-N([x, N(y)]), z)\\
&=&\underline{\mathcal{B}([N(x), N(y)], z)}+\mathcal{B}(N^{2}([x, y]), z)-\mathcal{B}(N([N(x), y]), z)-\mathcal{B}(N([x, N(y)]), z)\\
&=&-\mathcal{B}(N(y), [N(x), z])+\underline{\mathcal{B}([x, y], \widehat{N}^{2}(z))-\mathcal{B}([N(x), y], \widehat{N}(z))-\mathcal{B}([x, N(y)], \widehat{N}(z))}\\
&\stackrel{\eqref{eq:invariant}}{=}&-\mathcal{B}(y, \widehat{N}([N(x), z]))-\mathcal{B}(y, [x, \widehat{N}^{2}(z)])+\mathcal{B}(y, [N(x), \widehat{N}(z)])+\mathcal{B}(N(y), [x, \widehat{N}(z)])\\
&=&-\mathcal{B}(y, \widehat{N}([N(x), z]))-\mathcal{B}(y, [x, \widehat{N}^{2}(z)])+\mathcal{B}(y, [N(x), \widehat{N}(z)])+\mathcal{B}(y, \widehat{N}([x, \widehat{N}(z)]))\\
&=&\mathcal{B}(y, -\widehat{N}([N(x), z])-[x, \widehat{N}^{2}(z)]+[N(x), \widehat{N}(z)]+\widehat{N}([x, \widehat{N}(z)])).
\end{eqnarray*}}
Then, $$\widehat{N}([N(x), y])+[x, \widehat{N}^{2}(y)]-[N(x), \widehat{N}(y)]-\widehat{N}([x, \widehat{N}(y)])=0.$$
{\small
\begin{eqnarray*}
0&=&\mathcal{B}([N(x), N(y)]+N^{2}([x, y])-N([N(x), y])-N([x, N(y)]), z)\\
&=&\underline{\mathcal{B}([N(x), N(y)], z)}+\mathcal{B}(N^{2}([x, y]), z)-\mathcal{B}(N([N(x), y]), z)-\mathcal{B}(N([x, N(y)]), z)\\
&\stackrel{\eqref{eq:quadratic}}{=}&\mathcal{B}(N(x), [z, N(y)]-[N(y), z])+\underline{\mathcal{B}([x, y], \widehat{N}^{2}(z))-\mathcal{B}([N(x), y], \widehat{N}(z))-\mathcal{B}([x, N(y)], \widehat{N}(z))}\\
&\stackrel{\eqref{eq:quadratic}}{=}&\mathcal{B}(N(x), [z, N(y)]-[N(y), z])+\mathcal{B}(x, [\widehat{N}^{2}(z), y]-[y, \widehat{N}^{2}(z)])-\mathcal{B}(N(x), [\widehat{N}(z), y]-[y, \widehat{N}(z)])\\
&&-\mathcal{B}(x, [\widehat{N}(z), N(y)]-[N(y), \widehat{N}(z)])\\
&=&\mathcal{B}(x, \widehat{N}([z, N(y)]))-\mathcal{B}(x, \widehat{N}([N(y), z]))+\mathcal{B}(x, [\widehat{N}^{2}(z), y])-\mathcal{B}(x, [y, \widehat{N}^{2}(z)])-\mathcal{B}(x, \widehat{N}([\widehat{N}(z), y]))\\
&&+\mathcal{B}(x, \widehat{N}([y, \widehat{N}(z)]))-\mathcal{B}(x, [\widehat{N}(z), N(y)])+\mathcal{B}(x, [N(y), \widehat{N}(z)])\\
&=&\mathcal{B}(x, \widehat{N}([z, N(y)])+[\widehat{N}^{2}(z), y]-\widehat{N}([\widehat{N}(z), y])-[\widehat{N}(z), N(y)]).
\end{eqnarray*}}
Then, $$\widehat{N}([x, N(y)])+[\widehat{N}^{2}(x), y]-\widehat{N}([\widehat{N}(x), y])-[\widehat{N}(x), N(y)]=0.$$
By Lemma $\ref{lem:o}$, $((A^{*}, \mathcal{L}^{*}, \mathcal{L}^{*}-\mathcal{R}^{*}), \widehat{N}^{*})$ is a representation of $((A, [, ]), N)$.
Next, by \cite[Lemma 3]{KLWY}, $(A^{*}, \mathcal{L}^{*}, \mathcal{L}^{*}-\mathcal{R}^{*})$ is equivalent to $(A, \mathcal{L}, \mathcal{R})$ as representations of $(A, [, ])$. We define the linear map $\mathcal{B}^{\natural}:A\lr A^{*}$ by
$$\mathcal{B}^{\natural}(x)(y)=\langle \mathcal{B}^{\natural}(x),y\rangle =\mathcal{B}(x, y), $$
for all $x, y\in A$.
\begin{eqnarray*}
\mathcal{B}^{\natural}(N(x))(y)&=&\langle \mathcal{B}^{\natural}(N(x)), y\rangle=\mathcal{B}(N(x), y)=\mathcal{B}(x, \widehat{N}(y))\\
&=&\langle \mathcal{B}^{\natural}(x), \widehat{N}(y)\rangle=\langle \widehat{N}^{*}(\mathcal{B}^{\natural}(x)), y\rangle =\widehat{N}^{*}(\mathcal{B}^{\natural}(x))(y).
\end{eqnarray*}
Hence, $((A^{*}, \mathcal{L}^{*}, \mathcal{L}^{*}-\mathcal{R}^{*}), \widehat{N}^{*})$ is equivalent to $((A, \mathcal{L}, \mathcal{R}), N)$ as representations of $((A, [, ]), N)$.
Conversely, suppose that $\mathcal{B}^{\natural}: A\lr A^{*}$ is the linear isomorphism given by the equivalence between $((A^{*}, \mathcal{L}^{*}, \mathcal{L}^{*}-\mathcal{R}^{*}), S^{*})$ and $((A, \mathcal{L}, \mathcal{R}), N)$. Define a bilinear form $\mathcal{B}$ on $A$ by
$$\mathcal{B}(x, y)=\langle \mathcal{B}^{\natural}(x), y\rangle,$$
for all $x, y\in A$. Then by a similar argument as above, there exists a quadratic left-Alia algebra $(A, [, ], \mathcal{B})$ such that $S=\widehat{N}$.
\end{proof}

\begin{defi}\label{eq:ap}\cite{KLWY} Let $(A, [, ]_{A})$ and $(A^{*}, [, ]_{A^{*}})$ be two left-Alia algebras. Assume that there is a left-Alia algebra $(A\oplus A^{*}, [, ]_{A\oplus A^{*}})$ which contains $(A, [, ]_{A})$ and $(A^{*}, [, ]_{A^{*}})$ as left-Alia subalgebras. Suppose that the natural nondegenerate symmetric bilinear form $\mathcal{B}$, given by
\begin{eqnarray}\label{eq:naturalb}
\mathcal{B}(x+a^{*}, y+b^{*})=\langle a^{*}, y\rangle+\langle x, b^{*}\rangle, \forall~ x, y\in A, a^{*}, b^{*}\in A^{*},
\end{eqnarray}
is invariant on $(A\oplus A^{*}, [, ]_{A\oplus A^{*}})$, that is, $(A\oplus A^{*}, [, ]_{A\oplus A^{*}}, \mathcal{B})$ is a symmetric quadratic left-Alia algebra. Then, we say that $((A\oplus A^{*}, [, ]_{A\oplus A^{*}}, \mathcal{B}), A, A^{*})$ is a {\bf Manin triple of left-Alia algebras associated to $(A, [, ]_{A})$ and $(A^{*}, [, ]_{A^{*}})$}.
\end{defi}
Now we extend this notion to Nijenhuis left-Alia algebras.
\begin{defi}\label{de:aq} Let $((A, [, ]_{A}), N)$ and $((A^{*}, [, ]_{A^{*}}), S^{*})$ be two Nijenhuis left-Alia algebras. A {\bf Manin triple $((A\oplus A^{*}, [, ]_{A\oplus A^{*}}, N+S^{*}, \mathcal{B}), (A, N), (A^{*}, S^{*}))$ of Nijenhuis left-Alia algebras associated to $((A,[, ]_{A}),N)$ and $((A^{*}, [, ]_{A^{*}}), S^{*})$} is a Manin triple $((A\oplus A^{*}, [, ]_{A\oplus A^{*}}, \mathcal{B}), A, A^{*})$ of left-Alia algebras associated to $(A, [, ]_{A})$ and $(A^{*}, [, ]_{A^{*}})$ such that $((A\oplus A^{*}, [, ]_{A\oplus A^{*}}), N+S^{*})$ is a Nijenhuis left-Alia algebra.
\end{defi}
\begin{lem}\label{lem:ar} Let $((A\oplus A^{*}, [, ]_{A\oplus A^{*}}, N+S^{*}, \mathcal{B}), (A, N), (A^{*}, S^{*}))$ be a Manin triple of Nijenhuis left-Alia algebras associated to $((A,[, ]_{A}),N)$ and $((A^{*}, [, ]_{A^{*}}), S^{*})$.
\begin{enumerate}[(1)]
\item \label{it:lem:ar1} The adjoint linear map $\widehat{N+S^{*}}$ of $N+S^{*}$ with respect to $\mathcal{B}$ is $S+N^{*}$. Furthermore, $S+N^{*}$ is adjoint-\admt $((A\oplus A^{*}, [,]_{A\oplus A^{*}}), N+S^{*})$.
\item \label{it:lem:ar2} $S$ is adjoint-\admt $((A, [, ]_{A}), N)$.
\item \label{it:lem:ar3} $N^{*}$ is adjoint-\admt $((A^{*}, [, ]_{A^{*}}), S^{*})$.
\end{enumerate}
\end{lem}
\begin{proof} \ref{it:lem:ar1} For all $x, y \in A$, $a^{*}, b^{*}\in A^{*}$, we have
\begin{eqnarray*}
\mathcal{B}((N+S^{*})(x+a^{*}), y+b^{*})&=&\mathcal{B}(N(x)+S^{*}(a^{*}), y+b^{*})=\langle S^{*}(a^{*}), y\rangle+\langle N(x), b^{*}\rangle\\
&=&\langle a^{*}, S(y)\rangle+\langle x, N^{*}(b^{*})\rangle=\mathcal{B}(x+a^{*}, S(y)+N^{*}(b^{*}))\\
&=&\mathcal{B}(x+a^{*}, (S+N^{*})(y+b^{*})).
\end{eqnarray*}
Hence, $\widehat{N+S^{*}}=S+N^{*}$. Furthermore, by Proposition $\ref{pro:ao}$, $S+N^{*}$ is adjoint-\admt $((A\oplus A^{*}, [, ]_{A\oplus A^{*}}), N+S^{*})$.
\ref{it:lem:ar2} By Item \ref{it:lem:ar1}, $S+N^{*}$ is adjoint-\admt $((A\oplus A^{*}, [, ]_{A\oplus A^{*}}), N+S^{*})$. Then
\begin{eqnarray*}
&&(S+N^{*})[(N+S^{*})(x+a^{*}), y+b^{*}]_{A\oplus A^{*}}-(S+N^{*})[x+a^{*}, (S+N^{*})(y+b^{*})]_{A\oplus A^{*}}\\
&&+[(x+a^{*}), (S+N^{*})^{2}(y+b^{*})]_{A\oplus A^{*}}-[(N+S^{*})(x+a^{*}), (S+N^{*})(y+b^{*})]_{A\oplus A^{*}}=0,\\
&&(S+N^{*})[x+a^{*}, (N+S^{*})(y+b^{*})]_{A\oplus A^{*}}-(S+N^{*})[(S+N^{*})(x+a^{*}), y+b^{*}]_{A\oplus A^{*}}\\
&&+[(S+N^{*})^{2}(x+a^{*}), y+b^{*}]_{A\oplus A^{*}}-[(S+N^{*})(x+a^{*}), (N+S^{*})(y+b^{*})]_{A\oplus A^{*}}=0.
\end{eqnarray*}
Now taking $a^{*}=b^{*}=0$ in the equations above gives the admissibility of $S$ to $((A, [, ]_{A}), N)$.
\ref{it:lem:ar3} Taking $x=y=0$ in the equations above gives the admissibility of $N^{*}$ to $((A^{*}, [, ]_{A^{*}}), S^{*})$.
\end{proof}
More generally, we have
\begin{thm}\label{thm:at} Let $((A,[, ]_{A}),N)$ and $((A^{*}, [, ]_{A^{*}}), S^{*})$ be Nijenhuis left-Alia algebras. Then there is a Manin triple $((A\oplus A^{*}, [, ]_{A\oplus A^{*}}, N+S^{*}, \mathcal{B}), (A, N), (A^{*}, S^{*}))$ of $((A,[, ]_{A}),N)$ and $((A^{*}, [, ]_{A^{*}}), S^{*})$ if and only if $((A, N), (A^{*}, S^{*}), \mathcal{L_{A}^{*}}, \mathcal{L_{A}^{*}}-\mathcal{R_{A}^{*}}, \mathcal{L_{A^{*}}^{*}}, \mathcal{L_{A^{*}}^{*}}-\mathcal{R_{A^{*}}^{*}})$ is a matched pair of $((A,[, ]_{A}),N)$ and $((A^{*}, [, ]_{A^{*}}), S^{*})$.
\end{thm}
\begin{proof} ($\Longrightarrow$) If $((A\oplus A^{*}, [, ]_{A\oplus A^{*}}, N+S^{*},\mathcal{B}), (A, N), (A^{*}, S^{*}))$ is a Manin triple of $((A,[, ]_{A}),N)$ and $((A^{*}, [, ]_{A^{*}}), S^{*})$, then by Definition $\ref{de:aq}$, we know that $((A\oplus A^{*}, [, ]_{A\oplus A^{*}}, \mathcal{B}), A, A^{*})$ is a Manin triple of $(A, [, ]_{A})$ and $(A^{*}, [, ]_{A^{*}})$. By \cite[Theorem 5]{KLWY},
$((A, [, ]_{A}), (A^{*}, [, ]_{A^{*}}), \mathcal{L_{A}^{*}}, \mathcal{L_{A}^{*}}-\mathcal{R_{A}^{*}}, \mathcal{L_{A^{*}}^{*}}, \mathcal{L_{A^{*}}^{*}}-\mathcal{R_{A^{*}}^{*}})$ is a matched pair of $(A, [, ]_{A})$ and $(A^{*}, [, ]_{A^{*}})$. Furthermore, by Lemma $\ref{lem:ar}$, $((A^{*}, \mathcal{L_{A}^{*}}, \mathcal{L_{A}^{*}}-\mathcal{R_{A}^{*}}), S^{*})$ is a representation of $((A, [, ]_A), N)$ and $((A, \mathcal{L_{A^{*}}^{*}}, \mathcal{L_{A^{*}}^{*}}- \mathcal{R_{A^{*}}^{*}}), N)$ is a representation of $((A^{*}, [, ]_{A^{*}}), S^{*})$. Therefore, $((A, N), (A^{*}, S^{*}), \mathcal{L_{A}^{*}}, \mathcal{L_{A}^{*}}-\mathcal{R_{A}^{*}}, \mathcal{L_{A^{*}}^{*}}, \mathcal{L_{A^{*}}^{*}}-\mathcal{R_{A^{*}}^{*}})$ is a matched pair of $((A,[, ]_{A}),N)$ and $((A^{*}, [, ]_{A^{*}}), S^{*})$.
($\Longleftarrow$) If $((A, N), (A^{*}, S^{*}), \mathcal{L_{A}^{*}}, \mathcal{L_{A}^{*}}-\mathcal{R_{A}^{*}}, \mathcal{L_{A^{*}}^{*}}, \mathcal{L_{A^{*}}^{*}}-\mathcal{R_{A^{*}}^{*}})$ is a matched pair of $((A,[, ]_{A}),N)$ and $((A^{*}, [, ]_{A^{*}}), S^{*})$, then by Definition $\ref{de:am}$, $((A, [, ]_{A}), (A^{*}, [, ]_{A^{*}}), \mathcal{L_{A}^{*}}, \mathcal{L_{A}^{*}}-\mathcal{R_{A}^{*}}, \mathcal{L_{A^{*}}^{*}}, \mathcal{L_{A^{*}}^{*}}-\mathcal{R_{A^{*}}^{*}})$ is a matched pair of $(A, [, ]_{A})$ and $(A^{*}, [, ]_{A^{*}})$. By \cite[Theorem 5]{KLWY}, $((A\oplus A^{*}, [, ]_{A\oplus A^{*}}, \mathcal{B}), A, A^{*})$ is a Manin triple of $(A, [, ]_{A})$ and $(A^{*}, [, ]_{A^{*}})$. Furthermore, by Theorem $\ref{thm:an}$, $((A\oplus A^{*}, [, ]_{A\oplus A^{*}}), N+S^{*})$ is a Nijenhuis left-Alia algebra. Hence, $((A\oplus A^{*}, [, ]_{A\oplus A^{*}}, N+S^{*}, \mathcal{B}), (A, N), (A^{*}, S^{*}))$ is a Manin triple of $((A,[, ]_{A}),N)$ and $((A^{*}, [, ]_{A^{*}}), S^{*})$.
\end{proof}

\subsection{Nijenhuis left-Alia bialgebras} Recall from \cite{KLWY} that a {\bf left-Alia coalgebra} is a pair $(A, \D)$, such that $A$ is a vector space and $\D: A\lr A\o A$ (here we use Sweedler notation $\D(x)=x_{(1)}\o x_{(2)}, \forall~ x\in A$) is a linear map satisfying
\begin{eqnarray}
&&x_{(1)(2)}\o x_{(1)(1)}\o x_{(2)}+x_{(1)(1)}\o x_{(2)}\o x_{(1)(2)}+x_{(2)}\o x_{(1)(2)}\o x_{(1)(1)}\nonumber\\
&&\hspace{15mm}=x_{(1)(1)}\o x_{(1)(2)}\o x_{(2)}+x_{(1)(2)}\o x_{(2)}\o x_{(1)(1)}+x_{(2)}\o x_{(1)(1)}\o x_{(1)(2)}, \label{eq:cola}
\end{eqnarray}
for all $x\in A$. A {\bf left-Alia bialgebra} is a triple $(A, [, ], \D)$, such that $(A, [, ])$ is a left-Alia algebra, $(A, \D)$ is a left-Alia coalgebra and the following equation holds:
\begin{eqnarray}
(\tau-\id^{\o 2})(\D([x, y]-[y, x]))=(\tau-\id^{\o 2})((\mathcal{R}(y)\o \id)\D(x)-(\mathcal{R}(x)\o \id )\D(y)), \forall~ x, y\in A.\label{eq:labia}
\end{eqnarray}
Dual to Nijenhuis left-Alia algebra, we have
\begin{defi}\label{de:ah} A {\bf Nijenhuis left-Alia coalgebra} is a pair $((A, \D), S)$, where $(A, \D)$ is a left-Alia coalgebra and $S$ is a Nijenhuis operator on $(A, \D)$, i.e., the following identity holds:
\small{\begin{eqnarray}\label{eq:nlaca}
S(x_{(1)})\o S(x_{(2)})+S^{2}(x)_{(1)}\o S^{2}(x)_{(2)}=S(S(x)_{(1)})\o S(x)_{(2)}+S(x)_{(1)}\o S(S(x)_{(2)}), \forall~ x\in A.
\end{eqnarray}}
\end{defi}
For a finite-dimensional vector space $A$, $((A^{*}, [, ]_{A^{*}}), S^{*})$ is a Nijenhuis left-Alia algebra if and only if $((A, \D), S)$ is a Nijenhuis left-Alia coalgebra, where $\D: A\lr A\o A$ is the linear dual of $[, ]_{A^{*}}: A^{*}\o A^{*}\lr A^{*}$, that is, for all $x\in A$, $a^{*}, b^{*}\in A^{*}$,
\begin{equation*}
\langle a^{*}\o b^{*}, \D(x) \rangle=\langle [a^{*}, b^{*}]_{A^{*}}, x\rangle.
\end{equation*}
Moreover, for a linear map $N: A\lr A$, the condition that $N^{*}$ is adjoint-\admt the Nijenhuis left-Alia algebra $((A^{*}, [, ]_{A^{*}}), S^{*})$ can be rewritten in terms of $\D$ as
\begin{eqnarray}
&&S(N(x)_{(1)})\o N(x)_{(2)}+x_{(1)}\o N^{2}(x_{(2)})=S(x_{(1)})\o N(x_{(2)}) + N(x)_{(1)}\o N(N(x)_{(2)}),\label{eq:laacaadmi1}\\
&&N(x)_{(1)}\o S(N(x)_{(2)})+N^{2}(x_{(1)})\o x_{(2)}= N(x_{(1)})\o S(x_{(2)}) +N(N(x)_{(1)})\o N(x)_{(2)},\label{eq:laacaadmi2}
\end{eqnarray}
where $\forall~x\in A$.

\begin{defi}\label{de:au} A {\bf Nijenhuis left-Alia bialgebra} is a vector space $A$ together with linear maps $[, ]: A\o A\lr A$, $\D: A\lr A\o A$, $N, S : A\lr A$ such that
\begin{enumerate}[(1)]
\item\label{it:de:au1} $(A, [, ], \D)$ is a left-Alia bialgebra.
\item\label{it:de:au2} $((A, [, ]), N)$ is a Nijenhuis left-Alia algebra.
\item\label{it:de:au3} $((A, \D), S)$ is a Nijenhuis left-Alia coalgebra.
\item\label{it:de:au4} $S$ is adjoint-\admt $((A, [, ]), N)$, that is, Eqs.\eqref{eq:15}-\eqref{eq:16} hold.
\item\label{it:de:au5} $N^{*}$ is adjoint-\admt $((A^{*}, \D^{*}), S^{*})$, that is, Eqs.\eqref{eq:laacaadmi1}-\eqref{eq:laacaadmi2} hold.
\end{enumerate}
We denote this Nijenhuis left-Alia bialgebra by $((A, [, ], \D), N, S)$.
\end{defi}
\begin{ex}\label{ex:auu}
Let $(A, [, ])$ be a 4-dimensional left-Alia algebra with basis $\{e_{1}, e_{2}, e_{3}, e_{4}\}$ and the product $[, ]$ given by the following table
\begin{center}
\begin{tabular}{r|rrrr}
$[,]$ & $e_{1}$ & $e_{2}$ & $e_{3}$ & $e_{4}$\\
\hline
$e_{1}$ & $0$ & $0$ & $0$ & $0$\\
$e_{2}$ & $0$ & $0$ & $0$ & $0$\\
$e_{3}$ & $e_{1}$ & $0$ & $0$ & $0$\\
$e_{4}$ & $e_{3}$ & $0$ & $0$ & $0$\\
\end{tabular}.
\end{center}
Define
\begin{center}
$\begin{cases}
N(e_{1})=e_{1}\\
N(e_{2})=e_{1}+e_{2}\\
N(e_{3})=e_{3}\\
N(e_{4})=e_{4}\\
\end{cases}$.
\end{center}
Then $((A, [, ]), N)$ is a Nijenhuis left-Alia algebra. Set $\D=\D_r$ by 
\begin{center}
$\begin{cases}
\D_{r}(e_{1})=0\\
\D_{r}(e_{2})=0\\
\D_{r}(e_{3})=-e_{1}\o e_{2}\\
\D_{r}(e_{4})=-e_{3}\o e_{2}\\
\end{cases}$
\end{center}
and
\begin{center}
$\begin{cases}
S(e_{1})=e_{1}\\
S(e_{2})=e_{2}-e_{1}\\
S(e_{3})=e_{3}\\
S(e_{4})=e_{4}-e_{1}\\
\end{cases}$
\end{center}
Then $((A,\D), S)$ is a Nijenhuis left-Alia coalgebra. Furthermore, $((A, [, ], \D), N, S)$ is a Nijenhuis left-Alia bialgebra.
\end{ex}
\begin{thm}\label{thm:aw} Let $((A, [, ]), N)$ and $((A^{*}, \D^{*}), S^{*})$ be Nijenhuis left-Alia algebras. Then there is a Nijenhuis left-Alia bialgebra $((A, [, ], \D), N, S)$ if and only if $((A, N), (A^{*}, S^{*}), \mathcal{L_{A}^{*}}, \mathcal{L_{A}^{*}}-\mathcal{R_{A}^{*}}, \mathcal{L_{A^{*}}^{*}}, \mathcal{L_{A^{*}}^{*}}-\mathcal{R_{A^{*}}^{*}})$ is a matched pair of $((A, [, ]), N)$ and $((A^{*}, \D^{*}), S^{*})$.
\end{thm}
\begin{proof} ($\Longrightarrow$) If $((A, [, ], \D), N, S)$ is a Nijenhuis left-Alia bialgebra, then by Definition $\ref{de:au}$, we know that $(A, [, ], \D)$ is a left-Alia bialgebra and $S$, $N^{*}$ are adjoint-\admt $((A, [, ]), N)$ and $((A^{*}, \D^{*}), S^{*})$ respectively. It means that $((A^{*}, \mathcal{L_{A}^{*}}, \mathcal{L_{A}^{*}}- \mathcal{R_{A}^{*}}), S^{*})$ is a representation of $((A, [, ]), N)$ and $((A, \mathcal{L_{A^{*}}^{*}}, \mathcal{L_{A^{*}}^{*}}-\mathcal{R_{A^{*}}^{*}}), N)$ is a representation of $((A^{*}, \D^{*}), S^{*})$. By \cite[Theorem 6]{KLWY},
$((A, [, ]), (A^{*}, \D^{*}), \mathcal{L_{A}^{*}}, \mathcal{L_{A}^{*}}-\mathcal{R_{A}^{*}}, \mathcal{L_{A^{*}}^{*}}, \mathcal{L_{A^{*}}^{*}}-\mathcal{R_{A^{*}}^{*}})$ is a matched pair of $(A, [, ])$ and $(A^{*}, \D^{*})$ since $(A, [, ], \D)$ is a left-Alia bialgebra. Therefore, $((A, N), (A^{*}, S^{*}), \mathcal{L_{A}^{*}}, \mathcal{L_{A}^{*}}-\mathcal{R_{A}^{*}}, \mathcal{L_{A^{*}}^{*}}, \mathcal{L_{A^{*}}^{*}}-\mathcal{R_{A^{*}}^{*}})$ is a matched pair of $((A, [, ]), N)$ and $((A^{*}, \D^{*}), S^{*})$.
($\Longleftarrow$) If $((A, N), (A^{*}, S^{*}), \mathcal{L_{A}^{*}}, \mathcal{L_{A}^{*}}-\mathcal{R_{A}^{*}}, \mathcal{L_{A^{*}}^{*}}, \mathcal{L_{A^{*}}^{*}}-\mathcal{R_{A^{*}}^{*}})$ is a matched pair of $((A, [, ]), N)$ and $((A^{*}, \D^{*}), S^{*})$, then by Definition $\ref{de:am}$, $((A, [, ]), (A^{*}, \D^{*}), \mathcal{L_{A}^{*}}, \mathcal{L_{A}^{*}}-\mathcal{R_{A}^{*}}, \mathcal{L_{A^{*}}^{*}}, \mathcal{L_{A^{*}}^{*}}-\mathcal{R_{A^{*}}^{*}})$ is a matched pair of $(A, [, ])$ and $(A^{*}, \D^{*})$. By \cite[Theorem 6]{KLWY}, $(A, [, ], \D)$ is a left-Alia bialgebra. By Definition $\ref{de:am}$, $S$, $N^{*}$ are adjoint-\admt $((A, [, ]), N)$ and $((A^{*}, \D^{*}), S^{*})$ respectively. Hence, $((A, [, ], \D), N, S)$ is a Nijenhuis left-Alia bialgebra.
\end{proof}
\begin{thm}\label{thm:ax} Let $((A, [, ]), N)$ and $((A^{*}, [, ]_{A^{*}}(:=\D^*)), S^{*})$ be two Nijenhuis left-Alia algebras. Then the following conditions are equivalent:
\begin{enumerate}[(1)]
\item\label{it:thm:ax1} $((A, N), (A^{*}, S^{*}), \mathcal{L_{A}^{*}}, \mathcal{L_{A}^{*}}-\mathcal{R_{A}^{*}}, \mathcal{L_{A^{*}}^{*}}, \mathcal{L_{A^{*}}^{*}}-\mathcal{R_{A^{*}}^{*}})$ is a matched pair of $((A, [, ]), N)$ and $((A^{*}, [, ]_{A^{*}}), S^{*})$.
\item\label{it:thm:ax2} There is a Manin triple $((A\oplus A^{*}, [, ]_{A\oplus A^{*}}, N+S^{*}, \mathcal{B}), (A, N), (A^{*}, S^{*}))$ of $((A, [, ]), N)$ and $((A^{*}, [, ]_{A^{*}}), S^{*})$.
\item\label{it:thm:ax3} $((A, [, ], \D), N, S)$ is a Nijenhuis left-Alia bialgebra.
\end{enumerate}
\end{thm}
\begin{proof} It follows by Theorem \ref{thm:at} and Theorem \ref{thm:aw}.
\end{proof}

\section{$S$-admissible left-Alia Yang-Baxter equations}\label{se:ybe}
In this section, we mainly give some constructions of Nijenhuis left-Alia bialgebras via an element $r$ in $A\o A$.

\subsection{Triangular Nijenhuis left-Alia bialgebras}
\begin{defi}\label{de:u}\cite{KLY}~Let $(A, [, ])$ be a left-Alia algebra and $r=\sum_{i} a_{i}\o b_{i}\in A\o A$. We call
\begin{eqnarray}
Al(r)=\sum_{i, j}\Big([a_{i}, a_{j}]\o b_{i}\o b_{j}+a_{i}\o ([a_{j}, b_{i}]-[b_{i}, a_{j}])\o b_{j}-a_{i}\o a_{j}\o [b_{j}, b_{i}]\Big)=0.\label{eq:ybq}
\end{eqnarray}
a {\bf left-Alia Yang-Baxter equation in $(A, [, ])$}. And in this case we say $r$ is a {\bf solution of the left-Alia Yang-Baxter equation in $(A, [, ])$}.
\end{defi}
\begin{defi}\label{de:v}\cite{KLY} Let $(A, [, ])$ be a left-Alia algebra and $r=\sum_{i} a_{i}\o b_{i}\in A\o A$. Suppose that $r$ is an antisymmetric solution of the left-Alia Yang-Baxter equation. Then $(A, [, ], \D_{r})$ is a left-Alia bialgebra, where $\D_{r}: A\lr A\o A$ is given by
\begin{eqnarray}
\D_{r}(x)=\sum_{i}\big(([a_{i}, x]-[x, a_{i}])\o b_{i}-a_{i}\o [b_{i}, x]\big).\label{eq:comuti}
\end{eqnarray}
We say the resulting left-Alia bialgebra $(A, [, ], \D_{r})$ is {\bf triangular}.
\end{defi}

 Similar to Hopf algebra theory, triangular left-Alia bialgebras can be characterized as follows.
 
\begin{pro}\label{pro:w} Let $(A, [, ])$ be a left-Alia algebra and $r=\sum_{i} a_{i}\o b_{i}\in A\o A$. Then $r$ is a solution of left-Alia Yang-Baxter equation in $(A, [, ])$ if and only if
\begin{eqnarray}
(\id\o \D_{r})(r)=\sum_{i, j} -[a_{i}, a_{j}]\o b_{i}\o b_{j},\label{eq:dra}
\end{eqnarray}
where $\D_{r}$ is defined by Eq.\eqref{eq:comuti}.
If further, $r$ is antisymmetric (in the sense of $r+\tau(r)=0$), then $r$ is a solution of left-Alia Yang-Baxter equation in $(A, [, ])$ if and only if
\begin{eqnarray}\label{eq:dr}
(\D_{r}\o \id)(r)=\sum_{i, j} a_{i}\o a_{j}\o [b_{i}, b_{j}],
\end{eqnarray}
holds.
\end{pro}
\begin{proof} By Eq.\eqref{eq:comuti}, we have
\begin{eqnarray*}
(\id\o \D_{r})(r)=\sum_{i, j}\big(a_{i}\o ([a_{j}, b_{i}]-[b_{i}, a_{j}])\o b_{j}-a_{i}\o a_{j}\o [b_{j}, b_{i}]\big).
\end{eqnarray*}
Then Eq.\eqref{eq:dra} $\Leftrightarrow$ Eq.\eqref{eq:ybq}.
If $r$ is an antisymmetric element, we have
\begin{eqnarray*}
(\D_{r}\o \id)(r)
&=&\sum_{i, j}\big(([a_{j}, a_{i}]-[a_{i}, a_{j}])\o b_{j}\o b_{i}- a_{j}\o [b_{j}, a_{i}]\o b_{i}\big)\\
&=&\sum_{i, j}\big(-([a_{j}, b_{i}]-[b_{i}, a_{j}])\o b_{j}\o a_{i}+ a_{j}\o [b_{j}, b_{i}]\o a_{i}\big)\\
&=&\sum_{i, j}-\xi(a_{i}\o ([a_{j}, b_{i}]-[b_{i}, a_{j}])\o b_{j}-a_{i}\o a_{j}\o [b_{j}, b_{i}]),
\end{eqnarray*}
where $\xi(x\o y\o z)=y\o z\o x$ for all $x, y, z\in A$. Then Eq.\eqref{eq:dr} $\Leftrightarrow$ Eq.\eqref{eq:ybq}.
\end{proof}
Let $(A, [, ])$ be a left-Alia algebra and $r \in A\o A$. Set a map $$
h(x)=(\mathcal{R}-\mathcal{L})(x)\o \id-\id\o \mathcal{R}(x).
$$
Then the map $\D: A\lr A\o A$ is given by Eq.\eqref{eq:comuti} induces a left-Alia algebra structure on $A^{*}$ such that $(A, [, ], \D)$ is a left-Alia bialgebra \cite{KLY} if and only if, for all $x\in A$,
\begin{eqnarray}
&& (\id+\xi+\xi^{2})\bigg((\id\o \id \o \mathcal{R}(x))(\id\o \tau)Al(r)+\sum_{i}\Big(h([a_{i}, x]-[x, a_{i}])(r+\tau(r))\o b_{i}\nonumber\\
&&+\tau h(a_{i})(r+\tau(r))\o b_{i}+b_{i}\o (\mathcal{R}(a_{i})\o \id)h(x)(r+\tau(r))+(\id\o \mathcal{R}(b_{i})\o \mathcal{R}(x))\label{eq:laacob1}\\
&&(a_{i}\o (r+\tau(r)))+(\id\o \mathcal{R}(a_{i})\o \mathcal{R}(x))(b_{i}\o (r+\tau(r)))\Big)\bigg)=0,\nonumber
\end{eqnarray}
and
\begin{eqnarray}
&&\tau h([x, y]-[y, x])(r+\tau(r))-(\id\o \mathcal{R}(y))\tau h(x)(r+\tau(r))\label{eq:laacob2}\\
&&\hspace{46mm}+(\id\o \mathcal{R}(x))\tau h(y)(r+\tau(r))=0.\nonumber
\end{eqnarray}

\begin{pro}\label{pro:ba} Let $((A, [, ]), N)$ be an $S$-adjoint-admissible Nijenhuis left-Alia algebra and $r\in A\o A$. Define a linear map $\D: A\lr A\o A$ by Eq.\eqref{eq:comuti}. Then the following assertions hold.
\begin{enumerate}[(1)]
\item\label{it:pro:ba1} Eq.\eqref{eq:nlaca} holds if and only if, for all $x\in A$,
\begin{eqnarray}
&&((\mathcal{R}(S(x))-\mathcal{L}(S(x)))\o \id-S\circ(\mathcal{R}(x)-\mathcal{L}(x))\o \id)(N\o \id -\id \o S)(r)\nonumber\\
&&\hspace{30mm}+(\id\o \mathcal{R}(S(x))-\id\o S\circ \mathcal{R}(x))(S\o \id-\id\o N)(r)=0.\label{eq:41}
\end{eqnarray}
\item\label{it:pro:ba2} Eqs.\eqref{eq:laacaadmi1}-\eqref{eq:laacaadmi2} hold if and only if, for all $x\in A$,
\begin{eqnarray}
&&\big(\id\o (N\circ \mathcal{R}(x)-\mathcal{R}(N(x)))+(\mathcal{R}(N(x))-\mathcal{L}(N(x)))\o \id\nonumber\\
&&\hspace{20mm}+S\circ(\mathcal{R}(x)-\mathcal{L}(x))\o \id\big)(S\o \id-\id\o N)(r)\label{eq:42}\\
&&\hspace{30mm}+\big((\mathcal{L}(x)-\mathcal{R}(x))\circ S^{2}\o \id +(\mathcal{R}(x)-\mathcal{L}(x))\o N^{2}\big)(r)=0,\nonumber\\
&&\hspace{-10mm}\big(N\circ (\mathcal{R}(x)-\mathcal{L}(x))\o \id-(\mathcal{R}(N(x))-\mathcal{L}(N(x)))\o \id+\id\o (\mathcal{R}(N(x))\label{eq:43}\\
&&+S\circ \mathcal{R}(x))\big)(N\o \id-\id\o S)(r)+\big(\id\o\mathcal{R}(x)\circ S^{2}-N^{2}\o \mathcal{R}(x)\big)(r)=0,\nonumber
\end{eqnarray}
\end{enumerate}
\end{pro}
\begin{proof}\ref{it:pro:ba1} For all $x\in A$, we compute {\small
\begin{eqnarray*}
0&=&S(x_{(1)})\o S(x_{(2)})+S^{2}(x)_{(1)}\o S^{2}(x)_{(2)}-S(S(x)_{(1)})\o S(x)_{(2)}-S(x)_{(1)}\o S(S(x)_{(2)})\\
&\stackrel{\eqref{eq:comuti}}{=}&\sum_i \Big(S([a_{i}, x])\o S(b_{i})-S([x, a_{i}])\o S(b_{i})-S(a_{i})\o S([b_{i}, x])+\underline{[a_{i}, S^{2}(x)]\o b_{i}}\\
&&\underline{-[S^{2}(x), a_{i}]\o b_{i}-a_{i}\o [b_{i}, S^{2}(x)]-S([a_{i}, S(x)])\o b_{i}+S([S(x), a_{i}])\o b_{i}}\\
&&+S(a_{i})\o [b_{i}, S(x)]-[a_{i}, S(x)]\o S(b_{i})+[S(x), a_{i}]\o S(b_{i})+\underline{a_{i}\o S([b_{i}, S(x)])}\Big)\\
&\stackrel{\eqref{eq:15}\eqref{eq:16}}{=}&\sum_i\Big(S([a_{i}, x])\o S(b_{i})-S([x, a_{i}])\o S(b_{i})-S(a_{i})\o S([b_{i}, x])-\underline{S([N(a_{i}), x])\o b_{i}}\\
&&\underline{+S([x, N(a_{i})])\o b_{i}+a_{i}\o S([N(b_{i}), x])+[N(a_{i}), S(x)]\o b_{i}-[S(x), N(a_{i})]\o b_{i}}\\
&&+S(a_{i})\o [b_{i}, S(x)]-[a_{i}, S(x)]\o S(b_{i})+[S(x), a_{i}]\o S(b_{i})-\underline{a_{i}\o [N(b_{i}), S(x)]}\Big)\\
&=&((\mathcal{R}(S(x))-\mathcal{L}(S(x)))\o \id-S\circ(\mathcal{R}(x)-\mathcal{L}(x))\o \id)(N\o \id -\id \o S)(r)\\
&&+(\id\o \mathcal{R}(S(x))-\id\o S\circ \mathcal{R}(x))(S\o \id-\id\o N)(r).
\end{eqnarray*}}
Then, Eq.\eqref{eq:nlaca} holds if and only if Eq.\eqref{eq:41} holds.\smallskip

\ref{it:pro:ba2} Similarly, we have {\small
\begin{eqnarray*}
0&=&S(N(x)_{(1)})\o N(x)_{(2)}+x_{(1)}\o N^{2}(x_{(2)})-S(x_{(1)})\o N(x_{(2)})-N(x)_{(1)}\o N(N(x)_{(2)})\\
&\stackrel{\eqref{eq:comuti}}{=}&\sum_i \Big(\underbrace{S([a_{i}, N(x)])\o b_{i}-S[N(x), a_{i}]\o b_{i}}-S(a_{i})\o [b_{i}, N(x)]+[a_{i}, x]\o N^{2}(b_{i})\\
&&-[x, a_{i}]\o N^{2}(b_{i})-\underline{a_{i}\o N^{2}([b_{i}, x])}-S([a_{i}, x])\o N(b_{i})+S([x, a_{i}])\o N(b_{i})\\
&&+S(a_{i})\o N([b_{i}, x])-[a_{i}, N(x)]\o N(b_{i})+[N(x), a_{i}]\o N(b_{i})+\underline{a_{i}\o N([b_{i}, N(x)])}\Big)\\
&\stackrel{\eqref{eq:7}\eqref{eq:15}\eqref{eq:16}}{=}&\sum_i \Big(\underbrace{S([S(a_{i}), x])\o b_{i}-[S^{2}(a_{i}), x]\o b_{i}+[S(a_{i}), N(x)]\o b_{i}-S([x, S(a_{i})])\o b_{i}}\\
&&\underbrace{+[x, S^{2}(a_{i})]\o b_{i}-[N(x), S(a_{i})]\o b_{i}}-S(a_{i})\o [b_{i}, N(x)]+[a_{i}, x]\o N^{2}(b_{i})\\
&&-[x, a_{i}]\o N^{2}(b_{i})+\underline{a_{i}\o [N(b_{i}), N(x)]}-S([a_{i}, x])\o N(b_{i})+S([x, a_{i}])\o N(b_{i})\\
&&+S(a_{i})\o N([b_{i}, x])-[a_{i}, N(x)]\o N(b_{i})+[N(x), a_{i}]\o N(b_{i})\underline{-a_{i}\o N([N(b_{i}), x])}\Big)\\
&=&(\id\o (N\circ \mathcal{R}(x)-\mathcal{R}(N(x)))+(\mathcal{R}(N(x))-\mathcal{L}(N(x)))\o \id+S\circ(\mathcal{R}(x)-\mathcal{L}(x))\o \id)\\
&&(S\o \id-\id\o N)(r)+((\mathcal{L}(x)-\mathcal{R}(x))\circ S^{2}\o \id +(\mathcal{R}(x)-\mathcal{L}(x))\o N^{2})(r),
\end{eqnarray*}
and
\begin{eqnarray*}
0&=&N(x)_{(1)}\o S(N(x)_{(2)})+N^{2}(x_{(1)})\o x_{(2)}-N(x_{(1)})\o S(x_{(2)}) -N(N(x)_{(1)})\o N(x)_{(2)}\\
&\stackrel{\eqref{eq:comuti}}{=}&\sum_i \Big([a_{i}, N(x)]\o S(b_{i})- [N(x), a_{i}]\o S(b_{i})\underbrace{-a_{i}\o S([b_{i}, N(x)])}+ \underline{N^{2}([a_{i}, x])\o b_{i}}\\
&&\underline{- N^{2}([x, a_{i}])\o b_{i}}- N^{2}(a_{i})\o [b_{i}, x]\underline{-N([a_{i}, N(x)])\o b_{i}+N([N(x), a_{i}])\o b_{i}}\\
&&+ N(a_{i})\o [b_{i}, N(x)]- N([a_{i}, x])\o S(b_{i})+N([x, a_{i}])\o S(b_{i})+N(a_{i})\o S([b_{i}, x])\Big)\\
&\stackrel{\eqref{eq:7}\eqref{eq:16}}{=}&\sum_i \Big([a_{i}, N(x)]\o S(b_{i})- [N(x), a_{i}]\o S(b_{i})\underbrace{-a_{i}\o S([S(b_{i}), x])+a_{i}\o [S^{2}(b_{i}), x]}\\
&&\underbrace{- a_{i}\o [S(b_{i}), N(x)]}\underline{- [N(a_{i}), N(x)]\o b_{i}+ [N(x), N(a_{i})]\o b_{i}}-N^{2}(a_{i})\o [b_{i}, x]\\
&&\underline{+ N([N(a_{i}), x])\o b_{i}- N([x, N(a_{i})])\o b_{i}}+ N(a_{i})\o [b_{i}, N(x)]-N([a_{i}, x])\o S(b_{i})\\
&&+ N([x, a_{i}])\o S(b_{i})+ N(a_{i})\o S([b_{i}, x])\Big)\\
&=&\big(N\circ (\mathcal{R}(x)-\mathcal{L}(x))\o \id- (\mathcal{R}(N(x))-\mathcal{L}(N(x)))\o \id+\id\o (\mathcal{R}(N(x))+S\circ \mathcal{R}(x))\big)\\
&&(N\o \id-\id\o S)(r)+\big(\id\o\mathcal{R}(x)\circ S^{2}-N^{2}\o \mathcal{R}(x)\big)(r).
\end{eqnarray*}}
Then, Eq.\eqref{eq:laacaadmi1} holds if and only if Eq.\eqref{eq:42} holds, Eq.\eqref{eq:laacaadmi2} holds if and only if Eq.\eqref{eq:43} holds.
\end{proof}

\begin{rmk}\label{eq:bc}
\begin{enumerate}[(1)]
\item\label{it:eq:bc1} If $(N\o \id -\id \o S)(r)=0$, then $(\id\o\mathcal{R}(x)\circ S^{2}-N^{2}\o \mathcal{R}(x))(r)=0$.
\item\label{it:eq:bc2} If $(S\o \id-\id\o N)(r)=0$, then $((\mathcal{L}(x)-\mathcal{R}(x))\circ S^{2}\o \id +(\mathcal{R}(x)-\mathcal{L}(x))\o N^{2})(r)=0$.
\item\label{it:eq:bc3} If $r$ is antisymmetric, then $(N\o \id -\id \o S)(r)=0$ if and only if $(S\o \id-\id\o N)(r)=0$.
\end{enumerate}
\end{rmk}

\begin{thm}\label{thm:bd} Let $((A, [, ]), N)$ be an $S$-adjoint-admissible Nijenhuis left-Alia algebra and $r\in A\o A$. Define a linear map $\D$ by Eq.\eqref{eq:comuti}. Then $((A, [, ], \D), N, S)$ is a Nijenhuis left-Alia bialgebra if and only if Eqs.\eqref{eq:laacob1}-\eqref{eq:43} hold.
\end{thm}
\begin{cor}\label{cor:be} Let $((A, [, ]), N)$ be an $S$-adjoint-admissible Nijenhuis left-Alia algebra and $r\in A\o A$ be antisymmetric. Define a linear map $\D$ by Eq.\eqref{eq:comuti}. Then $((A, [, ], \D), N, S)$ is a Nijenhuis left-Alia bialgebra if Eq.(\ref{eq:ybq}) and the following equation hold:
\begin{eqnarray}
(S\o \id-\id\o N)(r)=0.\label{eq:salaeq2}
\end{eqnarray}
\end{cor}
\begin{rmk} Eq.\eqref{eq:salaeq2} can be seen as a compatible condition between left-Alia Yang-Baxter equation and Nijenhuis operators on a left-Alia algebra and a left-Alia coalgebra. This condition is obtained by the bialgebraic theory of Nijenhuis left-Alia algebras.\end{rmk}
\begin{defi}\label{de:bf} Let $((A, [, ]), N)$ be a Nijenhuis left-Alia algebra. Suppose that $r\in A\o A$ and $S:A\lr A$ is a linear map. Then Eq.\eqref{eq:ybq} together with Eq.\eqref{eq:salaeq2} is called an {\bf $S$-admissible left-Alia Yang-Baxter equation in $((A, [, ]), N)$}.\end{defi}
\begin{pro}\label{pro:bg} Let $((A, [, ]), N)$ be an $S$-adjoint-admissible Nijenhuis left-Alia algebra and $r\in A\o A$ an antisymmetric solution of $S$-admissible left-Alia Yang-Baxter equation in $((A, [, ]), N)$. Then $((A, [, ], \D), N, S)$ is a Nijenhuis left-Alia bialgebra, where the linear map $\D=\D_{r}$ is defined by Eq.\eqref{eq:comuti}. In this case, we call this Nijenhuis left-Alia bialgebra {\bf triangular}.
\end{pro}
\begin{proof} Straightforward.\end{proof}
\begin{ex}\label{ex:bh} Let $((A, [, ], \D), N, S)$ be a Nijenhuis left-Alia bialgebra given in Example \ref{ex:auu}. Then $((A, [, ], \D), N, S)$ is triangular with $r=e_{1}\o e_{2}-e_{2}\o e_{1}$.\end{ex}
\subsection{Relative Rota-Baxter operators on Nijenhuis left-Alia algebras}
\begin{defi}\label{de:bi}\cite{KLY} Let $(A, [, ])$ be a left-Alia algebra with a representation $(V, \ell, r)$. A linear map $T: V\lr A$ is called a {\bf relative Rota-Baxter operator of $(A, [, ])$ associated to $(V, \ell, r)$} if the following equation holds:
\begin{equation}
[T(u), T(v)]=T(\ell(T(u))(v)+r(T(v))(u)), \forall~ u, v\in V. \label{eq:weakr1}
\end{equation}
\end{defi}

For a vector space $A$, we can identify an element $r\in A\o A$ as a linear map $r^\#: A^*\lr A$ by
\begin{eqnarray}
r^\#(a^*):=\sum\limits_{i} \langle a^*, a_i\rangle b_i,~~\forall~a^*\in A^*.\label{eq:rsharp}
\end{eqnarray}
\begin{lem}\label{lem:bii}{\em \cite{KLY}} Let $(A, [, ])$ be a left-Alia algebra and $r\in A\o A$ antisymmetric. Then $r$ is a solution of the left-Alia Yang-Baxter equation in $(A, [, ])$ if and only if $r^\#$ defined in Eq.(\ref{eq:rsharp}) is a relative Rota-Baxter operator of $(A, [, ])$ associated to $(A^{*}, \mathcal{L_{A}^{*}},\mathcal{L_{A}^{*}}-\mathcal{R_{A}^{*}})$, i.e.,
\begin{eqnarray}
[r^\#(a^{*}), r^\#(b^{*})]=r^\#(\mathcal{L_{A}^{*}}(r^\#(a^{*}))b^{*}+(\mathcal{L_{A}^{*}}-\mathcal{R_{A}^{*}})(r^\#(b^{*}))a^{*}), ~\forall~ a^{*}, b^{*}\in A^{*}.\label{eq:54}
\end{eqnarray}
\end{lem}
\begin{thm}\label{thm:bj} Let $((A, [, ]), N)$ be a Nijenhuis left-Alia algebra, $r\in A\o A$ antisymmetric and $S: A\lr A$ a linear map. Then $r$ is a solution of the $S$-admissible left-Alia Yang-Baxter equation in $((A, [, ]), N)$ if and only if $r^\#$ is a relative Rota-Baxter operator of $(A, [, ])$ associated to $(A^{*}, \mathcal{L_{A}^{*}}, \mathcal{L_{A}^{*}}-\mathcal{R_{A}^{*}})$ and the equation below holds:
\begin{eqnarray}\label{eq:bj1}
N\circ r^\#=r^\#\circ S^{*}.
\end{eqnarray}
\end{thm}
\begin{proof} By Lemma \ref{lem:bii}, $r$ is a solution of left-Alia Yang-Baxter equation in $(A, [, ])$ if and only if Eq.\eqref{eq:54} holds. Moreover, for all $a^{*}, b^{*}\in A^{*}$, we have
\begin{eqnarray*}
&&\langle b^{*}, r^\#(S^{*}(a^{*})) \rangle=\langle S^{*}(a^{*})\o b^{*}, r\rangle =\langle a^{*}\o b^{*}, (S\o \id)(r) \rangle,\\
&&\langle b^{*}, N(r^\#(a^{*})) \rangle=\langle N^{*}(b^{*}), r^\#(a^{*}) \rangle =\langle a^{*}\o N^{*}(b^{*}), r \rangle=\langle a^{*}\o b^{*}, (\id\o N)r \rangle.
\end{eqnarray*}
Hence, Eq.\eqref{eq:bj1} holds if and only if Eq.\eqref{eq:salaeq2} holds.
\end{proof}
\begin{defi}\label{de:bk} Let $((A, [, ]), N)$ be a Nijenhuis left-Alia algebra, $(V, \ell, r)$ a representation of $(A, [, ])$ and $\a: V\lr V$ a linear map. A linear map $T: V\lr A$ is called a {\bf weak relative Rota-Baxter operator of $((A, [, ]), N)$ associated to $(V, \ell, r)$ and $\a$} if $T$ is a relative Rota-Baxter operator of $(A, [, ])$ associated to $(V, \ell, r)$ and the equation below holds:
\begin{eqnarray}
N\circ T=T\circ \a.\label{eq:weakr2}
\end{eqnarray}
If $((V, \ell, r), \a)$ is a representation of $((A, [, ]), N)$, then $T$ is called a {\bf relative Rota-Baxter operator of $((A, [, ]), N)$ associated to $((V, \ell, r), \a)$}.
\end{defi}
Theorem \ref{thm:bj} can be rewritten in terms of relative Rota-Baxter operators as follows.
\begin{cor}\label{cor:bl} Let $((A, [, ]), N)$ be a Nijenhuis left-Alia algebra, $r\in A\o A$ antisymmetric and $S: A\lr A$ a linear map. Then $r$ is a solution of $S$-admissible left-Alia Yang-Baxter equation in $((A, [, ]), N)$ if and only if $r^\#$ is a weak relative Rota-Baxter operator of $((A, [, ]), N)$ associated to $(A^{*}, \mathcal{L_{A}^{*}}, \mathcal{L_{A}^{*}}-\mathcal{R_{A}^{*}})$ and $S^{*}$. In addition, if $((A, [, ]), N)$ is an $S$-adjoint-admissible Nijenhuis left-Alia algebra, then $r$ is a solution of $S$-admissible left-Alia Yang-Baxter equation in $((A, [, ]), N)$ if and only if $r^\#$ is a relative Rota-Baxter operator of $((A, [, ]), N)$ associated to $((A^{*}, \mathcal{L_{A}^{*}}, \mathcal{L_{A}^{*}}-\mathcal{R_{A}^{*}}), S^{*})$.
\end{cor}

\begin{thm}\label{thm:bn} Let $((A, [, ]), N)$ be a Nijenhuis left-Alia algebra, $(V, \ell, r)$ a representation of $(A, [, ])$, $S: A\lr A$ and $\a, \b: V\lr V$ linear maps. Then the following conditions are equivalent.
\begin{enumerate}[(1)]
\item\label{it:thm:bn1} There is a Nijenhuis left-Alia algebra $(A\ltimes_{\ell, r} V, N+\a)$ such that the linear map $S+\b$ on $A\oplus V$ is adjoint-\admt $(A\ltimes_{\ell, r} V, N+\a)$.
\item\label{it:thm:bn2} There is a Nijenhuis left-Alia algebra $(A\ltimes_{\ell^{*}, \ell^{*}-r^{*}} V^{*}, N+\b^{*})$ such that the linear map $S+\a^{*}$ on $A\oplus V^{*}$ is adjoint-\admt $(A\ltimes_{\ell^{*}, \ell^{*}-r^{*}} V^{*}, N+\b^{*})$.
\item\label{it:thm:bn3} The following conditions are satisfied:
\begin{enumerate}
\item\label{it:thm:bna} $((V, \ell, r), \a)$ is a representation of $((A, [, ]), N)$, that is, Eqs.\eqref{eq:8} and \eqref{eq:9} hold.
\item\label{it:thm:bnb} $S$ is adjoint-\admt $((A, [, ]), N)$, that is, Eqs.\eqref{eq:15} and \eqref{eq:16} hold.
\item\label{it:thm:bnc} $\b$ is \admt $((A, [,]), N)$ on $(V, \ell, r)$, that is, Eqs.\eqref{eq:13} and \eqref{eq:14} hold.
\item\label{it:thm:bnd} The following equations hold, for all $x\in A$, $v\in V$,
\begin{eqnarray}
&\b(r(x)(\a(v)))+r(S^{2}(x))(v)=r(S(x))(\a(v))+\b(r(S(x))(v)),&\label{eq:59}\\
&\b(\ell(x)(\a(v)))+\ell(S^{2}(x))(v)=\ell(S(x))(\a(v))+\b(\ell(S(x))(v)).&\label{eq:60}
\end{eqnarray}
\end{enumerate}
\end{enumerate}
\end{thm}

\begin{proof}
\ref{it:thm:bn1}$\Leftrightarrow$\ref{it:thm:bn3}, By Proposition \ref{pro:n}, we have that $(A\ltimes_{\ell, r} V, N+\a)$ is a Nijenhuis left-Alia algebra if and only if $((V, \ell, r), \a)$ is a representation of $((A, [, ]), N)$. Let $x, y \in A$ and $u, v\in V$, we have
\begin{eqnarray*}
&&(S+\b)([(N+\a)(x+u), y+v]_{A\oplus V})+[x+u, (S+\b)^{2}(y+v)]_{A\oplus V}\\
&&-[(N+\a)(x+u), (S+\b)(y+v)]_{A\oplus V}-(S+\b)([x+u, (S+\b)(y+v)]_{A\oplus V})\\
&\stackrel{\eqref{eq:3}\eqref{eq:11}}{=}&S([N(x), y])+\b(\ell(N(x))(v))+\b(r(y)(\a(u)))+[x, S^{2}(y)]+\ell(x)(\b^{2}(v))\\
&&+r(S^{2}(y))(u)-[N(x), S(y)]-\ell(N(x))(\b(v))-r(S(y))(\a(u))-S([x, S(y)])\\
&&-\b(\ell(x)(\b(v)))-\b(r(S(y))(u)).
\end{eqnarray*}
Then $(S+\b)([(N+\a)(x+u), y+v]_{A\oplus V})+[x+u, (S+\b)^{2}(y+v)]_{A\oplus V}=[(N+\a)(x+u), (S+\b)(y+v)]_{A\oplus V}+(S+\b)([x+u, (S+\b)(y+v)]_{A\oplus V})$ if and only if Eqs.\eqref{eq:13}, \eqref{eq:15} and \eqref{eq:59} hold.
Similarly, we can get that $(S+\b)([x+u, (N+\a)(y+v)]_{A\oplus V})+[(S+\b)^{2}(x+u), y+v]_{A\oplus V}=[(S+\b)(x+u), (N+\a)(y+v)]_{A\oplus V}+(S+\b)([(S+\b)(x+u), y+v]_{A\oplus V})$ if and only if Eqs.\eqref{eq:14}, \eqref{eq:16} and \eqref{eq:60} hold.
\ref{it:thm:bn2}$\Leftrightarrow$\ref{it:thm:bn3}. In Item \ref{it:thm:bn1}, take
$$V=V^{*},\ell=\ell^{*},r=\ell^{*}-r^{*}, \a=\b^{*},\b=\a^{*}.$$
Then from the above equivalence between \ref{it:thm:bn1} and \ref{it:thm:bn3}, we have \ref{it:thm:bn2} holds if and only if \ref{it:thm:bna}-\ref{it:thm:bnc} hold and the following equations hold:
\begin{eqnarray}
&\a^{*}(r^{*}(x)(\b^{*}(u^{*})))+r^{*}(S^{2}(x))(u^{*})-r^{*}(S(x))(\b^{*}(u^{*}))
-\a^{*}(r^{*}(S(x))(u^{*}))=0,&\label{eq:61}\\
&\a^{*}(\ell^{*}(x)(\b^{*}(u^{*})))+\ell^{*}(S^{2}(x))(u^{*})
-\ell^{*}(S(x))(\b^{*}(u^{*}))-\a^{*}(\ell^{*}(S(x))(u^{*}))=0.&\label{eq:62}
\end{eqnarray}
For all $x\in A$, $v\in V$ and $u^{*}\in V^{*}$, we have
\begin{eqnarray*}
&&\big\langle \a^{*}(r^{*}(x)(\b^{*}(u^{*})))+r^{*}(S^{2}(x))(u^{*})-r^{*}(S(x))(\b^{*}(u^{*}))
-\a^{*}(r^{*}(S(x))(u^{*})), v\big\rangle\\
&\stackrel{\eqref{eq:repdual}}{=}&\big\langle u^{*}, -\b(r(x)(\a(v)))-r(S^{2}(x))(v)+\b(r(S(x))(v))+r(S(x))(\a(v))\big\rangle.\\
&&\big\langle \a^{*}(\ell^{*}(x)(\b^{*}(u^{*})))+\ell^{*}(S^{2}(x))(u^{*})-\ell^{*}(S(x))(\b^{*}(u^{*}))
-\a^{*}(\ell^{*}(S(x))(u^{*})), v\big\rangle\\
&\stackrel{\eqref{eq:repdual}}{=}&\big\langle u^{*}, -\b(\ell(x)(\a(v)))-\ell(S^{2}(x))(v)+\b(\ell(S(x))(v))+\ell(S(x))(\a(v))\big\rangle.
\end{eqnarray*}
Hence, Eq.\eqref{eq:61} holds if and only if Eq.\eqref{eq:59} holds, and Eq.\eqref{eq:62} holds if and only if Eq.\eqref{eq:60} holds. 
\end{proof}

\begin{lem}\label{lem:bm}{\em \cite{KLY}} Let $(A, [, ])$ be a left-Alia algebra with a representation $(V, \ell, r)$. Let $T: V\lr A$ be a linear map and $T_{\sharp}\in V^{*}\o A\subset (A\oplus V^{*})\o (A\oplus V^{*})$ given by
\begin{equation*}
\langle T_{\sharp}, u\o a^{*}\rangle:=\langle T(u), a^{*}\rangle, \forall~ u\in V, a^{*}\in A^{*}.
\end{equation*}
Then $r=T_{\sharp}-\tau(T_{\sharp})$ is an antisymmetric solution of the left-Alia Yang-Baxter equation in $A\ltimes_{\ell^{*}, \ell^{*}-r^{*}} V^{*}$ if and only if $T$ is a relative Rota-Baxter operator of $(A, [, ])$ associated to $(V, \ell, r)$.
\end{lem}

\begin{thm}\label{thm:bo} Let $((A, [, ]), N)$ be a $\b$-admissible Nijenhuis left-Alia algebra with respect to the representation $(V, \ell, r)$, $S: A\lr A$, $\a: V\lr V$ and $T: V\lr A$ be linear maps.
\begin{enumerate}[(1)]
\item\label{it:thm:bo1} $r=T_{\sharp}-\tau(T_{\sharp})$ is an antisymmetric solution of the $(S+\a^*)$-admissible left-Alia Yang-Baxter equation in $(A\ltimes_{\ell^{*}, \ell^{*}-r^{*}} V^{*}, N+\b^{*})$ if and only if $T$ is a weak relative Rota-Baxter operator of $(A, [, ])$ associated to $(V, \ell, r)$ and $\a$, and $T\circ \b=S\circ T$.
\item\label{it:thm:bo2} Assume that $((V, \ell, r), \a)$ is a representation of $((A, [, ]), N)$. If $T$ is a relative Rota-Baxter operator of $((A, [, ]), N)$ associated to $((V, \ell, r), \a)$ and $T\circ \b=S\circ T$, then $r=T_{\sharp}-\tau(T_{\sharp})$ is an antisymmetric solution of the $(S+\a^*)$-admissible left-Alia Yang-Baxter equation in $(A\ltimes_{\ell^{*}, \ell^{*}-r^{*}} V^{*}, N+\b^{*})$. In addition, if $((A, [, ]), N)$ is $S$-adjoint-admissible and Eqs.\eqref{eq:59} and \eqref{eq:60} hold, then there is a Nijenhuis left-Alia bialgebra $((A\ltimes_{\ell^{*}, \ell^{*}-r^{*}} V^{*}, \D), N+\b^{*}, S+\a^{*})$, where the linear map $\D=\D_{r}$ is defined by Eq.\eqref{eq:comuti} with $r=T_{\sharp}-\tau(T_{\sharp})$.
\end{enumerate}
\end{thm}
\begin{proof} \ref{it:thm:bo1} Let $\{e_{1}, e_{2}, ..., e_{n}\}$ be a basis of $V$ and $\{e_{1}^{*}, e_{2}^{*},..., e_{n}^{*}\}$ be the dual basis. Then
\begin{eqnarray*}
&&T_{\sharp}=\sum_{i} e_{i}^{*}\o T(e_{i})\in V^{*}\o A\subset (A\oplus V^{*})\o (A\oplus V^{*}),\\
&&r=T_{\sharp}-\tau(T_{\sharp})=\sum_{i} (e_{i}^{*}\o T(e_{i})-T(e_{i})\o e_{i}^{*}).
\end{eqnarray*}
Note that
\begin{eqnarray*}
&&((S+\a^{*})\o \id)(r)=\sum_{i} (\a^{*}(e_{i}^{*})\o T(e_{i})-S(T(e_{i}))\o e_{i}^{*}),\\
&&(\id\o(N+\b^{*}))(r)=\sum_{i} (e_{i}^{*}\o N(T(e_{i}))-T(e_{i})\o \b^{*}(e_{i}^{*})).
\end{eqnarray*}
Further,
\begin{eqnarray*}
&&\sum_{i} T(e_{i})\o \b^{*}(e_{i}^{*})=\sum_{i} T(\b(e_{i}))\o e_{i}^{*},\\
&&\sum_{i} \a^{*}(e_{i}^{*})\o T(e_{i})=\sum_{i} e_{i}^{*}\o T(\a(e_{i})).
\end{eqnarray*}
Therefore $((S+\a^{*})\o \id)(r)=(\id\o(N+\b^{*}))(r)$ if and only if $T\circ \b=S\circ T$ and $N\circ T=T\circ \a$. Hence the conclusion follows by Lemma \ref{lem:bm}.
\ref{it:thm:bo2} It follows from Item \ref{it:thm:bo1} and Theorem \ref{thm:bn}.
\end{proof}
\section{Application to special left-Alia algebras}\label{se:special}
An associative D-bialgebra introduced by Zhelyabin in \cite{Zhe} (also called balanced infinitesimal bialgebra by Aguiar in \cite{Ag00b} or antisymmetric infinitesimal bialgebra by Bai in \cite{Bai1}) is a triple $(A, \cdot, \d)$, where $(A, \cdot)$ is an associative algebra, $(A, \d)$ is a coassociative coalgebra such that two compatible conditions hold. The following is a special case.
\begin{defi} A {\bf commutative and cocommutative associative D-bialgebra} is a triple $(A, \cdot, \d)$, where the pair $(A, \cdot)$ is a commutative associative algebra, and the pair $(A, \d)$ is a cocommutative coassociative coalgebra, such that, for all $a, b\in A$,
\begin{equation}\label{eq:asi}
\d(a \cdot b)=a_{[1]}\cdot b\o a_{[2]}+b_{[1]}\o a\cdot b_{[2]},
\end{equation}
where we write $\d(a)=a_{[1]}\o a_{[2]}, \forall~a\in A$.
\end{defi}
Now we consider the case of {\bf special left-Alia bialgebra}.

\begin{thm}\label{thm:ajj} Let $(A, \cdot, \d)$ be a commutative and cocommutative associative D-bialgebra with linear maps $f, g, F, G: A\lr A$. Then there is a left-Alia bialgebra $(A, [,]_{(f, g)}, \D_{(F, G)})$, where $[,]_{(f, g)}$ and $\D_{(F, G)}$ are given by Eqs.(\ref{eq:special}) and (\ref{eq:cosplaa}) respectively, if and only if for all $x, y\in A$, {\small
\begin{eqnarray}
&&f(y)_{[1]}\o F(x\cdot f(y)_{[2]})-F(x\cdot f(y)_{[2]})\o f(y)_{[1]}+F(x_{[2]})\o g(x_{[1]}\cdot y)\nonumber\\
&&-g(x_{[1]}\cdot y)\o F(x_{[2]})+G(x)_{[2]}\o G(x)_{[1]}\cdot f(y)-G(x)_{[1]}\cdot f(y)\o G(x)_{[2]}\nonumber\\
&&+G(x)_{[2]}\o g(G(x)_{[1]}\cdot y)-g(G(x)_{[1]}\cdot y)\o G(x)_{[2]}+F(y\cdot f(x)_{[2]})\o f(x)_{[1]}\label{eq:pro:ajj1}\\
&&-f(x)_{[1]}\o F(y\cdot f(x)_{[2]})-F(y_{[2]})\o g(y_{[1]}\cdot x)+g(y_{[1]}\cdot x)\o F(y_{[2]})\nonumber\\
&&-G(y)_{[2]}\o G(y)_{[1]}\cdot f(x)+G(y)_{[1]}\cdot f(x)\o G(y)_{[2]}-G(y)_{[2]}\o g(G(y)_{[1]}\cdot x)\nonumber\\
&&+g(G(y)_{[1]}\cdot x)\o G(y)_{[2]}=0.\nonumber
\end{eqnarray}}
In this case, $(A, [, ]_{(f, g)}, \D_{(F, G)})$ is called a {\bf special left-Alia bialgebra with respect to $(f, g, F, G)$}.
\end{thm}
\begin{proof} It is easy to see that $(A, [, ]_{(f, g)})$ is a special left-Alia algebra with respect to $(A, \cdot, f, g)$ and $(A, \D_{(F, G)})$ is a special left-Alia coalgebra with respect to $(A, \d, F, G)$. Now we prove that $(A, [, ]_{(f, g)}, \D_{(F, G)})$ is a left-Alia bialgebra. In fact, for all $x, y\in A$, we have
{\small\begin{eqnarray*}
&&(\tau-\id^{\o 2})((\mathcal{R}(y)\o \id)\D_{(F, G)}(x)-(\mathcal{R}(x)\o \id )\D_{(F, G)}(y))-(\tau-\id^{\o 2})(\D_{(F, G)}([x, y]_{(f, g)}-[y, x]_{(f, g)}))\\
&\stackrel{\eqref{eq:cosplaa}\eqref{eq:special}}{=}&(\tau-\id^{\o 2})(x_{[1]}\cdot f(y)\o F(x_{[2]})+g(x_{[1]}\cdot y)\o F(x_{[2]})+G(x)_{[1]}\cdot f(y)\o G(x)_{[2]}\\
&&+g(G(x)_{[1]}\cdot y)\o G(x)_{[2]}-y_{[1]}\cdot f(x)\o F(y_{[2]})-g(y_{[1]}\cdot x)\o F(y_{[2]})-G(y)_{[1]}\cdot f(x)\\
&&\o G(y)_{[2]}-g(G(y)_{[1]}\cdot x)\o G(y)_{[2]})- (\tau-\id^{\o 2})((x\cdot f(y))_{[1]}\o F((x\cdot f(y))_{[2]})\\
&&+G(y\cdot f(x))_{[1]}\o G(y\cdot f(x))_{[2]}-(y\cdot f(x))_{[1]}\o F((y\cdot f(x))_{[2]})-G(y\cdot f(x))_{[1]}\o G(y\cdot f(x))_{[2]})\\
&\stackrel{\eqref{eq:asi}}{=}&\underline{F(x_{[2]})\o x_{[1]}\cdot f(y)}+F(x_{[2]})\o g(x_{[1]}\cdot y)+G(x)_{[2]}\o G(x)_{[1]}\cdot f(y)+G(x)_{[2]}\o g(G(x)_{[1]}\cdot y)\\
&&\underline{-F(y_{[2]})\o y_{[1]}\cdot f(x)}-F(y_{[2]})\o g(y_{[1]}\cdot x)-G(y)_{[2]}\o G(y)_{[1]}\cdot f(x)-G(y)_{[2]}\o g(G(y)_{[1]}\cdot x)\\
&&\underline{-x_{[1]}\cdot f(y)\o F(x_{[2]})}-g(x_{[1]}\cdot y)\o F(x_{[2]})-G(x)_{[1]}\cdot f(y)\o G(x)_{[2]}-g(G(x)_{[1]}\cdot y)\o G(x)_{[2]}\\
&&\underline{+y_{[1]}\cdot f(x)\o F(y_{[2]})}+g(y_{[1]}\cdot x)\o F(y_{[2]})+G(y)_{[1]}\cdot f(x)\o G(y)_{[2]}+g(G(y)_{[1]}\cdot x)\o G(y)_{[2]}\\
&&\underline{-F(x_{[2]})\o x_{[1]}\cdot f(y)}-F(x\cdot f(y)_{[2]})\o f(y)_{[1]}\underline{+x_{[1]}\cdot f(y)\o F(x_{[2]})}+f(y)_{[1]}\o F(x\cdot f(y)_{[2]})\\
&&\underline{+F(y_{[2]})\o y_{[1]}\cdot f(x)}+F(y\cdot f(x)_{[2]})\o f(x)_{[1]}\underline{-y_{[1]}\cdot f(x)\o F(y_{[2]})}-f(x)_{[1]}\o F(y\cdot f(x)_{[2]})\\
&=&-F(x\cdot f(y)_{[2]})\o f(y)_{[1]}+f(y)_{[1]}\o F(x\cdot f(y)_{[2]})+F(x_{[2]})\o g(x_{[1]}\cdot y)-g(x_{[1]}\cdot y)\o F(x_{[2]})\\
&&+G(x)_{[2]}\o G(x)_{[1]}\cdot f(y)-G(x)_{[1]}\cdot f(y)\o G(x)_{[2]}+G(x)_{[2]}\o g(G(x)_{[1]}\cdot y)-g(G(x)_{[1]}\cdot y)\\
&&\o G(x)_{[2]}+F(y\cdot f(x)_{[2]})\o f(x)_{[1]}-f(x)_{[1]}\o F(y\cdot f(x)_{[2]})-F(y_{[2]})\o g(y_{[1]}\cdot x)+g(y_{[1]}\cdot x)\\
&&\o F(y_{[2]})-G(y)_{[2]}\o G(y)_{[1]}\cdot f(x)+G(y)_{[1]}\cdot f(x)\o G(y)_{[2]}-G(y)_{[2]}\o g(G(y)_{[1]}\cdot x)\\
&&+g(G(y)_{[1]}\cdot x)\o G(y)_{[2]}.
\end{eqnarray*}}
Hence, $(A, [,]_{(f, g)}, \D_{(F, G)})$ is a left-Alia bialgebra if and only if Eq.\eqref{eq:pro:ajj1} holds.
\end{proof}
\begin{cor}\label{cor:ajjh} Let $(A, \cdot, \d)$ be a commutative cocommutative associative D-bialgebra with linear maps $f, g: A\lr A$. Then there is a special left-Alia bialgebra $(A, [, ]_{(f, g)}, \D_{(g, f)})$ with respect to $(f, g, g, f)$.
\end{cor}
\begin{proof} Since $(A, \cdot, \d)$ is a commutative cocommutative associative D-bialgebra, we have $\d(x\cdot y)=\d(y\cdot x)$, $\forall~x, y\in A$, that is,
\begin{eqnarray}
&x_{[1]}\cdot y\o x_{[2]}+y_{[1]}\o x\cdot y_{[2]}=y_{[1]}\cdot x\o y_{[2]}+x_{[1]}\o y\cdot x_{[2]}.&\label{eq:asi1}
\end{eqnarray}
Next we check that Eq.\eqref{eq:pro:ajj1} holds as follows.
{\small \begin{eqnarray*}
\hbox{LHS of Eq.}\eqref{eq:pro:ajj1}&=&\underline{f(y)_{[1]}\o g(x\cdot f(y)_{[2]})}\underline{-g(x\cdot f(y)_{[2]})\o f(y)_{[1]}}+g(x_{[2]})\o g(x_{[1]}\cdot y)\\
&&-g(x_{[1]}\cdot y)\o g(x_{[2]})+f(x)_{[2]}\o f(x)_{[1]}\cdot f(y)-f(x)_{[1]}\cdot f(y)\o f(x)_{[2]}\\
&&\underbrace{+f(x)_{[2]}\o g(f(x)_{[1]}\cdot y)}\underbrace{-g(f(x)_{[1]}\cdot y)\o f(x)_{[2]}}+\underbrace{g(y\cdot f(x)_{[2]})\o f(x)_{[1]}}\\
&&\underbrace{-f(x)_{[1]}\o g(y\cdot f(x)_{[2]})}-g(y_{[2]})\o g(y_{[1]}\cdot x)+g(y_{[1]}\cdot x)\o g(y_{[2]})\\
&&-f(y)_{[2]}\o f(y)_{[1]}\cdot f(x)+f(y)_{[1]}\cdot f(x)\o f(y)_{[2]}\underline{-f(y)_{[2]}\o g(f(y)_{[1]}\cdot x)}\\
&&\underline{+g(f(y)_{[1]}\cdot x)\o f(y)_{[2]}}\\
&=&\underline{g(x_{[2]})\o g(x_{[1]}\cdot y)-g(x_{[1]}\cdot y)\o g(x_{[2]})}\\
&&\underbrace{+f(x)_{[2]}\o f(x)_{[1]}\cdot f(y)-f(x)_{[1]}\cdot f(y)\o f(x)_{[2]}}\\
&&\underline{-g(y_{[2]})\o g(y_{[1]}\cdot x)+g(y_{[1]}\cdot x)\o g(y_{[2]})}\\
&&\underbrace{-f(y)_{[2]}\o f(y)_{[1]}\cdot f(x)+f(y)_{[1]}\cdot f(x)\o f(y)_{[2]}}\\
&\stackrel{\eqref{eq:asi1}}{=}&0.
\end{eqnarray*}}
as desired.
\end{proof}

\begin{thm}\label{thm:q} Let $(A, \cdot)$ be a commutative associative algebra with linear maps $f, g: A\lr A$. If $((A, \cdot), f)$ is a Nijenhuis associative algebra and $f\circ g=g\circ f$, then $((A, [, ]_{(f, g)}), f)$ is a Nijenhuis special left-Alia algebra.
\end{thm}
\begin{proof} For any $x, y\in A$, we have
\begin{eqnarray*}
&&[f(x), f(y)]_{(f, g)}+f^{2}([x, y]_{(f, g)})-f([f(x), y]_{(f, g)})-f([x, f(y)]_{(f, g)})\\
&\stackrel{\eqref{eq:special}}{=}&f(x)\cdot f(f(y))+g(f(x)\cdot f(y))+f^{2}(x\cdot f(y))+f^{2}(g(x\cdot y))\\
&&-f(f(x)\cdot f(y))-f(g(f(x)\cdot y))-f(x\cdot f(f(y)))-f(g(x\cdot f(y)))\\
&=&0,
\end{eqnarray*}
finishing the proof.
\end{proof}
Similarly, we have
\begin{cor} \label{cor:q} Let $(A, \cdot)$ be a commutative associative algebra with linear maps $f, g: A\lr A$. If $((A, \cdot), g)$ is a Nijenhuis associative algebra and $f\circ g=g\circ f$, then $((A, [, ]_{(f, g)}), g)$ is a Nijenhuis special left-Alia algebra.
\end{cor}
Dual to \cite[Theorem 6.2]{Dz} and Theorem \ref{thm:q}, one has
\begin{pro}\label{pro:ahhh} Let $(A, \d)$ be a cocommutative coassociative coalgebra with linear maps $F, G: A\lr A$. Then there is a special left-Alia coalgebra $(A, \D_{(F, G)})$ with respect to $(A, \d, F, G)$. If $((A, \d), F)$ is a Nijenhuis coassociative coalgebra and $F\circ G=G\circ F$, then $((A, \D_{(F, G)}), F)$ is a Nijenhuis special left-Alia coalgebra.
\end{pro}
Likewise, one can obtain
\begin{cor}\label{cor:ahhh} Let $(A, \d)$ be a cocommutative coassociative coalgebra with linear maps $F, G: A\lr A$. If $((A, \d), G)$ is a Nijenhuis coassociative coalgebra and $F\circ G=G\circ F$, then $((A, \D_{(F, G)}), G)$ is a Nijenhuis special left-Alia coalgebra.
\end{cor}

\begin{thm}\label{thm:ajjj} Let $f, g, F, G: A\lr A$ be linear maps such that any two of them commute, let $((A, \cdot, \d), f, F)$ be a commutative and cocommutative Nijenhuis associative D-bialgebra and Eq.\eqref{eq:pro:ajj1} holds. Then $((A, [, ]_{(f, g)}, \D_{(F, G)}), f, F)$ is a Nijenhuis special left-Alia bialgebra.
\end{thm}

\begin{proof} By Theorem \ref{thm:ajj}, we have $(A, [, ]_{(f, g)}, \D_{(F, G)})$ is a special left-Alia bialgebra. By Theorem \ref{thm:q}, we have $((A, [, ]_{(f, g)}), f)$ is a Nijenhuis special left-Alia algebra and by Proposition \ref{pro:ahhh}, $((A, \D_{(F, G)}), F)$ is a Nijenhuis special left-Alia coalgebra. Next we only need to check that Eqs.(\ref{eq:15}), (\ref{eq:16}), (\ref{eq:laacaadmi1}) and (\ref{eq:laacaadmi2}) hold for $[, ]_{(f, g)}, \D_{(F, G)}$ and $f, F$. Since $((A, \cdot, \d), f, F)$ is a commutative cocommutative Nijenhuis associative D-bialgebra, for all $x, y\in A$, one gets
\begin{eqnarray}
&F(f(x)\cdot y)+x\cdot F^2(y)=f(x)\cdot F(y)+F(x\cdot F(y)),&\label{eq:pduql}\\
&F(f(x)_{[1]})\o f(x)_{[2]}+x_{[1]}\o f^2(x_{[2]})=F(x_{[1]})\o f(x_{[2]})+f(x)_{[1]}\o f(f(x)_{[2]}).&\label{eq:pduqld}
\end{eqnarray}
Then we compute as follows.
\begin{eqnarray*}
&&F([f(x), y]_{(f, g)})+[x, F^{2}(y)]_{(f, g)}-F([x, F(y)]_{(f, g)})-[f(x), F(y)]_{(f, g)}\\
&\stackrel{\eqref{eq:special}}{=}&F(f(x)\cdot f(y))+F(g(f(x)\cdot y))+x\cdot f(F^{2}(y))+g(x\cdot F^{2}(y))-F(x\cdot f(F(y)))\\
&&-F(g(x\cdot F(y)))-f(x)\cdot f(F(y))-g(f(x)\cdot F(y))\\
&=&F(f(x)\cdot f(y))+g(F(f(x)\cdot y))+x\cdot F^{2}(f(y))+g(x\cdot F^{2}(y))-F(x\cdot F(f(y)))\\
&&-g(F(x\cdot F(y)))-f(x)\cdot F(f(y))-g(f(x)\cdot F(y))\\
&\stackrel{\eqref{eq:pduql}}{=}&0.
\end{eqnarray*}
Then Eq.(\ref{eq:15}) holds for $[, ]_{(f, g)}$ and $f, F$. Similarly, Eq. (\ref{eq:16}) holds for $[, ]_{(f, g)}$ and $f, F$.
\begin{eqnarray*}
&&F(f(x)_{(1)})\o f(x)_{(2)}+x_{(1)}\o f^{2}(x_{(2)})-F(x_{(1)})\o f(x_{(2)}) - f(x)_{(1)}\o f(f(x)_{(2)})\\
&\stackrel{\eqref{eq:cosplaa}}{=}&F(f(x)_{[1]})\o F(f(x)_{[2]})+F(G(f(x))_{[1]})\o G(f(x))_{[2]} +x_{[1]}\o f^{2}(F(x_{[2]}))\\
&&+G(x)_{[1]}\o f^{2}(G(x)_{[2]})-F(x_{[1]})\o f(F(x_{[2]}))-F(G(x)_{[1]})\o f(G(x)_{[2]})\\
&&-f(x)_{[1]}\o f(F(f(x)_{[2]}))- G(f(x))_{[1]}\o f(G(f(x))_{[2]})\\
&=&F(f(x)_{[1]})\o F(f(x)_{[2]})+F(f(G(x))_{[1]})\o f(G(x))_{[2]} +x_{[1]}\o F(f^{2}(x_{[2]}))\\
&&+G(x)_{[1]}\o f^{2}(G(x)_{[2]})-F(x_{[1]})\o f(F(x_{[2]}))-F(G(x)_{[1]})\o f(G(x)_{[2]})\\
&&-f(x)_{[1]}\o F(f(f(x)_{[2]}))- f(G(x))_{[1]}\o f(f(G(x))_{[2]})\\
&\stackrel{\eqref{eq:pduqld}}{=}&0.
\end{eqnarray*}
Then Eq.(\ref{eq:laacaadmi1}) holds for $\D_{(F, G)}$ and $f, F$. Similarly, Eq. (\ref{eq:laacaadmi2}) holds for $\D_{(F, G)}$ and $f, F$.
Hence, $((A, [, ]_{(f, g)}, \D_{(F, G)}), f, F)$ is a Nijenhuis special left-Alia bialgebra.
\end{proof}
\begin{cor}\label{cor:ajjjh} Let $((A, \cdot, \d), f, g)$ be a commutative cocommutative Nijenhuis associative D-bialgebra. If $f\circ g=g\circ f$, then $((A, [, ]_{(f, g)}, \D_{(g, f)}), f, g)$ is a Nijenhuis special left-Alia bialgebra.
\end{cor}
\begin{proof} By Corollary \ref{cor:ajjh}, we have $(A,[, ]_{(f, g)},\D_{(g, f)})$ is a special left-Alia bialgebra. Since $f\circ g=g\circ f$, by Theorem \ref{thm:ajjj} we have $((A, [, ]_{(f, g)}, \D_{(g, f)}), f, g)$ is a Nijenhuis special left-Alia bialgebra.
\end{proof}
\begin{cor} Let $f, g, F, G: A\lr A$ be linear maps such that any two of them commute.
\begin{enumerate}[(1)]
\item If $((A, \cdot, \d), f, G)$ is a commutative cocommutative Nijenhuis associative D-bialgebra and Eq.\eqref{eq:pro:ajj1} holds, then $((A, [, ]_{(f, g)}, \D_{(F, G)}), f, G)$ is a Nijenhuis special left-Alia bialgebra.
\item If $((A, \cdot, \d), g, F)$ is a commutative cocommutative Nijenhuis associative D-bialgebra and Eq.\eqref{eq:pro:ajj1} holds, then $((A, [, ]_{(f, g)}, \D_{(F, G)}), g, F)$ is a Nijenhuis special left-Alia bialgebra.
\item If $((A, \cdot, \d), g, G)$ is a commutative cocommutative Nijenhuis associative D-bialgebra and Eq.\eqref{eq:pro:ajj1} holds, then $((A, [, ]_{(f, g)}, \D_{(F, G)}), g, G)$ is a Nijenhuis special left-Alia bialgebra.
\item If $((A, \cdot, \d), f, f)$ is a commutative cocommutative Nijenhuis associative D-bialgebra, then $((A, [, ]_{(f, g)}, \D_{(g, f)}), f, f)$ is a Nijenhuis special left-Alia bialgebra.
\item If $((A, \cdot, \d), g, f)$ is a commutative cocommutative Nijenhuis associative D-bialgebra, then $((A, [, ]_{(f, g)}, \D_{(g, f)}), g, f)$ is a Nijenhuis special left-Alia bialgebra.
\item If $((A, \cdot, \d), g, g)$ is a commutative cocommutative Nijenhuis associative D-bialgebra, then $((A, [, ]_{(f, g)}, \D_{(g, f)}), g, g)$ is a Nijenhuis special left-Alia bialgebra.
\end{enumerate}
\end{cor}
\section{A construction of Nijenhuis left-Alia algebras}\label{se:nijcon} In order to give a construction of Nijenhuis left-Alia algebras, we need to introduce some notions and related results.
\begin{defi}\label{de:x} Let $(A, \D)$ be a left-Alia coalgebra and $\omega\in (A\o A)^{*}$ a bilinear form. The equation
\begin{eqnarray}\label{eq:dybq}
\omega(x_{(1)}, y)\omega(x_{(2)}, z)+\omega(x, y_{(2)})\omega(y_{(1)}, z)=\omega(x, y_{(1)})\omega(y_{(2)}, z)+\omega(x, z_{(2)})\omega(y, z_{(1)}),
\end{eqnarray}
where $x, y, z\in A$, is called a {\bf co-left-Alia Yang-Baxter equation} in $(A, \D)$.
\end{defi}

\begin{thm}\label{thm:y} Let $(A, \D)$ be a left-Alia coalgebra and $\omega\in ( A\o A)^{*}$. Suppose that $\omega$ is a skew-symmetric (in the sense of $\omega(x, y)+\omega(y, x)=0$) solution of the co-left-Alia Yang-Baxter equation in $(A, \D)$. Then $(A, [, ]_{\omega}, \D)$ is a left-Alia bialgebra, where $[, ]_{\omega}: A\o A\lr A$ is given by
\begin{eqnarray}\label{eq:muti}
[x, y]_{\omega}=x_{(2)}\omega(x_{(1)}, y)-x_{(1)}\omega(x_{(2)}, y)-y_{(2)}\omega(x, y_{(1)}), \forall~ x,y\in A.
\end{eqnarray}
This bialgebra is called a {\bf dual triangular left-Alia bialgebra}.
\end{thm}
\begin{proof} We only check that the symmetric Jacobi identity holds for $[, ]_{\omega}$ as follows, the rest is straightforward. For all $x, y, z\in A$, we have
\begin{eqnarray*}
[x, y]_{\omega}-[y, x]_{\omega}
&=&x_{(2)}\omega(x_{(1)}, y)-x_{(1)}\omega(x_{(2)}, y)-y_{(2)}\omega(x, y_{(1)})-y_{(2)}\omega(y_{(1)}, x)\\
&&+y_{(1)}\omega(y_{(2)}, x)+x_{(2)}\omega(y, x_{(1)})\\
&=&-x_{(1)}\omega(x_{(2)}, y)+y_{(1)}\omega(y_{(2)}, x).
\end{eqnarray*}
So
\begin{eqnarray*}
&&[[x, y]_{\omega}-[y, x]_{\omega}, z]_{\omega}+[[y, z]_{\omega}-[z, y]_{\omega}, x]_{\omega}+[[z, x]_{\omega}-[x, z]_{\omega}, y]_{\omega}\\
&=&-\omega(x_{(2)}, y)\omega(x_{(1)(1)}, z)x_{(1)(2)}+\omega(x_{(2)}, y)\omega(x_{(1)(2)}, z)x_{(1)(1)}\\
&&+\underline{\omega(x_{(2)}, y)\omega(x_{(1)}, z_{(1)})z_{(2)}}\\
&&+\omega(y_{(2)}, x)\omega(y_{(1)(1)}, z)y_{(1)(2)}-\omega(y_{(2)}, x)\omega(y_{(1)(2)}, z)y_{(1)(1)}-\underline{\omega(y_{(2)}, x)\omega(y_{(1)}, z_{(1)})z_{(2)}}\\
&&-\omega(y_{(2)}, z)\omega(y_{(1)(1)}, x)y_{(1)(2)}+\omega(y_{(2)}, z)\omega(y_{(1)(2)}, x)y_{(1)(1)}+\underbrace{\omega(y_{(2)}, z)\omega(y_{(1)}, x_{(1)})x_{(2)}}\\
&&+\omega(z_{(2)}, y)\omega(z_{(1)(1)}, x)z_{(1)(2)}-\omega(z_{(2)}, y)\omega(z_{(1)(2)}, x)z_{(1)(1)}-\underbrace{\omega(z_{(2)}, y)\omega(z_{(1)}, x_{(1)})x_{(2)}}\\
&&-\omega(z_{(2)}, x)\omega(z_{(1)(1)}, y)z_{(1)(2)}+\omega(z_{(2)}, x)\omega(z_{(1)(2)}, y)z_{(1)(1)}+\underleftrightarrow{\omega(z_{(2)}, x)\omega(z_{(1)}, y_{(1)})y_{(2)}}\\
&&+\omega(x_{(2)}, z)\omega(x_{(1)(1)}, y)x_{(1)(2)}-\omega(x_{(2)}, z)\omega(x_{(1)(2)}, y)x_{(1)(1)}-\underleftrightarrow{\omega(x_{(2)}, z)\omega(x_{(1)}, y_{(1)})y_{(2)}}\\
&\stackrel{\eqref{eq:dybq}}{=}&\underbrace{-\omega(x_{(2)}, y)\omega(x_{(1)(1)}, z)x_{(1)(2)}+\omega(x_{(2)}, y)\omega(x_{(1)(2)}, z)x_{(1)(1)}-\omega(y, x_{(1)(2)})\omega(x_{(1)(1)}, z)x_{(2)}}\\
&&\underleftrightarrow{+\omega(y_{(2)}, x)\omega(y_{(1)(1)}, z)y_{(1)(2)}-\omega(y_{(2)}, x)\omega(y_{(1)(2)}, z)y_{(1)(1)}+\omega(z, y_{(1)(1)})\omega(y_{(1)(2)}, z)y_{(2)}}\\
&&\underleftrightarrow{-\omega(y_{(2)}, z)\omega(y_{(1)(1)}, x)y_{(1)(2)}+\omega(y_{(2)}, z)\omega(y_{(1)(2)}, x)y_{(1)(1)}-\omega(z, y_{(1)(2)})\omega(y_{(1)(1)}, x)y_{(2)}}\\
&&\underline{+\omega(z_{(2)}, y)\omega(z_{(1)(1)}, x)z_{(1)(2)}-\omega(z_{(2)}, y)\omega(z_{(1)(2)}, x)z_{(1)(1)}-\omega(x, z_{(1)(2)})\omega(z_{(1)(1)}, y)z_{(2)}}\\
&&\underline{-\omega(z_{(2)}, x)\omega(z_{(1)(1)}, y)z_{(1)(2)}+\omega(z_{(2)}, x)\omega(z_{(1)(2)}, y)z_{(1)(1)}+\omega(x, z_{(1)(1)})\omega(z_{(1)(2)}, y)z_{(2)}}\\
&&\underbrace{+\omega(x_{(2)}, z)\omega(x_{(1)(1)}, y)x_{(1)(2)}-\omega(x_{(2)}, z)\omega(x_{(1)(2)}, y)x_{(1)(1)}+\omega(y, x_{(1)(1)})\omega(x_{(1)(2)}, z)x_{(2)}}\\
&\stackrel{\eqref{eq:cola}}{=}&0,
\end{eqnarray*}
completing the proof.
\end{proof} 

Dual to Proposition \ref{pro:w}, one can get 
\begin{pro}\label{pro:z} Let $(A, \D)$ be a left-Alia coalgebra and $\omega\in( A\o A)^{*}$ a bilinear form. Then $\omega$ is a solution of the co-left-Alia Yang-Baxter equation in $(A, \D)$, if and only if
\begin{eqnarray}\label{eq:mw}
\omega(x, [y, z]_{\omega})=-\omega(x_{(1)}, y)\omega(x_{(2)}, z),
\end{eqnarray}
where $[, ]_{\omega}$ is defined by Eq.\eqref{eq:muti}.
If further $\omega\in( A\o A)^{*}$ is skew-symmetric, then $\omega$ is a solution of the co-left-Alia Yang-Baxter equation in $(A, \D)$, if and only if
\begin{eqnarray}\label{eq:mwr}
\omega([x, y]_{\omega}, z)=\omega(x, z_{(1)})\omega(y, z_{(2)}),
\end{eqnarray}
holds, where $x, y, z\in A$.
\end{pro}

\begin{defi}\label{de:abb}\cite{KLY} Let $(A, [, ])$ be a left-Alia algebra and $\omega\in (A\o A)^{*}$ a skew-symmetric bilinear form. Assume that
\begin{eqnarray}\label{eq:symplectic}
\omega([y, z]-[z, y], x)=\omega([x, y], z)-\omega([x, z], y),\forall~ x, y, z\in A.
\end{eqnarray}
Then we call $\om$ a {\bf symplectic structure} on $(A, [, ])$, and $(A, [, ])$ a {\bf symplectic left-Alia algebra} denoted by $(A, [, ], \omega)$.
\end{defi}

\begin{rmk}\label{rmk:de:abb} If $r^\#$ defined by Eq.(\ref{eq:rsharp}) is bijective, then we call $r$ {\bf nondegenerate}. By \cite[Proposition 5.4]{KLY}, we know that $r$ is an antisymmetric nondegenerate solution of the left-Alia Yang-Baxter equation in $(A, [, ])$ if and only if $\om_r$ is a nondegenerate symplectic structure on $(A, [, ])$, where $\om_r$ is defined by
\begin{eqnarray}
\om_r(x, y):=\langle {r^\#}^{-1}(x), y \rangle.\label{eq:omegar}
\end{eqnarray}
\end{rmk}

\begin{pro}\label{pro:ab} Let $(A, [, ]_{\omega}, \D)$ be a dual triangular left-Alia bialgebra. Then $(A, [, ]_{\omega}, \omega)$ is a symplectic left-Alia algebra.
\end{pro}
\begin{proof} For all $x, y, z\in A$, we have
\begin{eqnarray*}
&&\omega([y, z]_{\omega}-[z, y]_{\omega}, x)-\omega([x, y]_{\omega}, z)+\omega([x, z]_{\omega}, y)\\
&\stackrel{\eqref{eq:mwr}}{=}&\omega(-y_{(1)}\omega(y_{(2)}, z)+z_{(1)}\omega(z_{(2)}, y), x)-\omega(y, z_{(2)})\omega(x, z_{(1)})+\omega(z, y_{(2)})\omega(x, y_{(1)})\\
&=&0,
\end{eqnarray*}
as desired.
\end{proof}

Now we present the main result in this section. 
\begin{thm}\label{thm:ac} Let $(A, [, ],\omega)$ be a symplectic left-Alia algebra and $r\in A\o A$. If $(A, [, ], \D_{r})$ is a triangular left-Alia bialgebra together with the comultiplication $\D_{r}$ given by Eq.\eqref{eq:comuti}, and further $(A, [, ]_{\omega}, \D_{r})$ is a dual triangular left-Alia bialgebra together with the multiplication $[, ]_{\omega}$ given by Eq.\eqref{eq:muti}, then $((A, [, ]), N)$ is a Nijenhuis left-Alia algebra, where $N:A\lr A$ is given by
\begin{equation}\label{eq:ac2}
N(x):=\sum_{i}\omega(x, a_{i})b_{i}, \forall~ x\in A.
\end{equation}
\end{thm}
\begin{proof} For all $x, y, z\in A$, by Eqs.\eqref{eq:comuti} and \eqref{eq:dybq}, we have {\small
\begin{eqnarray*}
0&=&\omega(x_{(1)}, y)\omega(x_{(2)}, z)+\omega(x, y_{(2)})\omega(y_{(1)}, z)-\omega(x, y_{(1)})\omega(y_{(2)}, z)-\omega(x, z_{(2)})\omega(y, z_{(1)})\\
&=&\sum_{i}\Big(\underbrace{\omega([a_{i}, x], y)\omega(b_{i}, z)-\omega([x, a_{i}], y)\omega(b_{i}, z)}-\omega(a_{i}, y)\omega([b_{i}, x], z)\\
&&\underbrace{+\omega(x, b_{i})\omega([a_{i}, y], z)-\omega(x, b_{i})\omega([y, a_{i}], z)}-\omega(x, [b_{i}, y])\omega(a_{i}, z)\\
&&-\omega(x, [a_{i}, y])\omega(b_{i}, z)+\omega(x, [y, a_{i}])\omega(b_{i}, z)+\omega(x, a_{i})\omega([b_{i}, y], z)\\
&&-\omega(x, b_{i})\omega(y, [a_{i}, z])+\omega(x, b_{i})\omega(y, [z, a_{i}])+\omega(x, [b_{i}, z])\omega(y, a_{i})\Big)\\
&\stackrel{\eqref{eq:symplectic}}{=}&\sum_{i}\Big(\underbrace{\omega([y, a_{i}], x)\omega(b_{i}, z)}-\omega([y, x], a_{i})\omega(b_{i}, z)-\omega(a_{i}, y)\omega([b_{i}, x], z)\\
&&\underbrace{+\omega(x, b_{i})\omega([z, a_{i}], y)}-\omega(x, b_{i})\omega([z, y], a_{i})-\omega(x, [b_{i}, y])\omega(a_{i}, z)\\
&&-\omega(x, [a_{i}, y])\omega(b_{i}, z)+\underbrace{\omega(x, [y, a_{i}])\omega(b_{i}, z)}+\omega(x, a_{i})\omega([b_{i}, y], z)\\
&&-\omega(x, b_{i})\omega(y, [a_{i}, z])+\underbrace{\omega(x, b_{i})\omega(y, [z, a_{i}])}+\omega(x, [b_{i}, z])\omega(y, a_{i})\Big)\\
&=&\sum_{i}\Big(-\omega([y, x], a_{i})\omega(b_{i}, z)-\omega(a_{i}, y)\omega([b_{i}, x], z)-\omega(x, b_{i})\omega([z, y], a_{i})-\omega(x, [b_{i}, y])\omega(a_{i}, z)\\
&&-\omega(x, [a_{i}, y])\omega(b_{i}, z)+\underbrace{\omega(x, a_{i})\omega([b_{i}, y], z)}-\omega(x, b_{i})\omega(y, [a_{i}, z])+\omega(x, [b_{i}, z])\omega(y, a_{i})\Big)\\
&\stackrel{\eqref{eq:dybq}}{=}&\sum_{i}\Big(-\omega([y, x], a_{i})\omega(b_{i}, z)-\omega(a_{i}, y)\omega([b_{i}, x], z)-\omega(x, b_{i})\omega([z, y], a_{i})-\omega(x, [b_{i}, y])\omega(a_{i}, z)\\
&&-\omega(x, [a_{i}, y])\omega(b_{i}, z)+\underbrace{\omega(x, a_{i})\omega([y, b_{i}], z)+\omega(x, a_{i})\omega([z, b_{i}], y)-\omega(x, a_{i})\omega([z, y], b_{i})}\\
&&-\omega(x, b_{i})\omega(y, [a_{i}, z])+\omega(x, [b_{i}, z])\omega(y, a_{i})\Big)\\
&=&\sum_{i}\Big(-\omega([y, x], a_{i})\omega(b_{i}, z)-\omega(a_{i}, y)\omega([b_{i}, x], z)-\underbrace{\omega(x, b_{i})\omega([z, y], a_{i})-\omega(x, [b_{i}, y])\omega(a_{i}, z)}\\
&&\underbrace{-\omega(x, [a_{i}, y])\omega(b_{i}, z)}+\omega(x, a_{i})\omega([y, b_{i}], z)+\omega(x, a_{i})\omega([z, b_{i}], y)\underbrace{-\omega(x, a_{i})\omega([z, y], b_{i})}\\
&&-\omega(x, b_{i})\omega(y, [a_{i}, z])+\omega(x, [b_{i}, z])\omega(y, a_{i})\Big)\\
&=&\sum_{i}\Big(-\omega([y, x], a_{i})\omega(b_{i}, z)-\omega(a_{i}, y)\omega([b_{i}, x], z)+\omega(x, a_{i})\omega([y, b_{i}], z)+\omega(x, a_{i})\omega([z, b_{i}], y)\\
&&-\omega(x, b_{i})\omega(y, [a_{i}, z])+\omega(x, [b_{i}, z])\omega(y, a_{i})\Big)\\
&=&\sum_{i}\Big(-\omega([y, x], a_{i})\omega(b_{i}, z)+\omega([b_{i}, x],z)\omega(y, a_{i})+\omega([y, b_{i}],z)\omega(x, a_{i})-\omega(y, [z, b_{i}])\omega(x, a_{i})\\
&&+\omega(y, [b_{i}, z])\omega(x, a_{i})-\omega(y, b_{i})\omega(x, [a_{i}, z])\Big).
\end{eqnarray*}
Thus
\begin{eqnarray}\label{eq:ac1-1}
&&\sum_{i}\Big(-\omega([y, x], a_{i})\omega(b_{i}, z)+\omega([b_{i}, x], z)\omega(y, a_{i})+\omega([y, b_{i}], z)\omega(x, a_{i})\nonumber\\
&&\qquad-\omega(y, [z, b_{i}])\omega(x, a_{i})+\omega(y, [b_{i}, z])\omega(x, a_{i})-\omega(y, b_{i})\omega(x, [a_{i}, z])\Big)=0.
\end{eqnarray}}
Now we check that $N$ is a Nijenhuis operator on $(A, [, ])$.
\begin{eqnarray*}
&&[N(x), N(y)]+N^{2}([x, y])-N([N(x), y])-N([x, N(y)])\\
&=&\sum_{i,j}\Big(\omega(x, a_{i})\omega(y, a_{j})[b_{i}, b_{j}]+\omega([x, y], a_{i})\omega(b_{i}, a_{j})b_{j}-\omega(x, a_{i})\omega([b_{i}, y], a_{j})b_{j}\\
&&-\omega([x, b_{i}], a_{j})\omega(y, a_{i})b_{j}\Big)\\
&\stackrel{\eqref{eq:ac1-1}}{=}&\sum_{i,j}\Big(\omega(x, a_{i})\omega(y, a_{j})[b_{i}, b_{j}]+\omega([x, y], a_{i})\omega(b_{i}, a_{j})b_{j}-\omega(x, a_{i})\omega([b_{i}, y], a_{j})b_{j}\\
&&-\omega([x, b_{i}], a_{j})\omega(y, a_{i})b_{j}-\omega([x, y], a_{i})\omega(b_{i}, a_{j})b_{j}+\omega(x, a_{i})\omega([b_{i}, y], a_{j})b_{j}\\
&&+\omega([x, b_{i}], a_{j})\omega(y, a_{i})b_{j}-\omega(x, [a_{j}, b_{i}])\omega(y, a_{i})b_{j}+\omega(x, [b_{i}, a_{j}])\omega(y, a_{i})b_{j}\\
&&-\omega(x, b_{i})\omega(y, [a_{i}, a_{j}])b_{j}\Big)\\
&=&\sum_{i,j}\Big(\omega(x, a_{i})\omega(y, a_{j})[b_{i}, b_{j}]-\omega(x, [a_{j}, b_{i}])\omega(y, a_{i})b_{j}+\omega(x, [b_{i}, a_{j}])\omega(y, a_{i})b_{j}\\
&&-\omega(x, b_{i})\omega(y, [a_{i}, a_{j}])b_{j}\Big)\\
&\stackrel{\eqref{eq:ybq}}{=}&0,
\end{eqnarray*}
finishing the proof.
\end{proof}

\begin{ex} \label{ah} Let $(A, [, ], \D_{r})$ be a triangular left-Alia bialgebra with $r$ nondegenerate, that is, $r$ is an antisymmetric nondegenerate solution of the left-Alia Yang-Baxter equation in $(A, [, ])$, then by Remark \ref{rmk:de:abb}, $(A, [,], \om_r)$ is a symplectic left-Alia algebra, where $\om_r$ is given in Eq.(\ref{eq:omegar}).
Next we check that $(A, [, ]_{\omega_r}, \D_{r})$ is a dual triangular left-Alia bialgebra, that is, $\om_r$ is a skew-symmetric solution of the co-left-Alia Yang-Baxter equation in $(A, \D_{r})$. In fact, for all $x, y, z\in A$, we have
{\small \begin{eqnarray*}
&&\omega_{r}(x_{(1)}, y)\omega_{r}(x_{(2)}, z)+\omega_{r}(x, y_{(2)})\omega_{r}(y_{(1)}, z)-\omega_{r}(x, y_{(1)})\omega_{r}(y_{(2)}, z)-\omega_{r}(x, z_{(2)})\omega_{r}(y, z_{(1)})\\
&=&\sum_{i}\Big(\underbrace{\omega([a_{i}, x], y)\omega(b_{i}, z)-\omega([x, a_{i}], y)\omega(b_{i}, z)}-\omega(a_{i}, y)\omega([b_{i}, x], z)\\
&&\underbrace{+\omega(x, b_{i})\omega([a_{i}, y], z)-\omega(x, b_{i})\omega([y, a_{i}], z)}-\omega(x, [b_{i}, y])\omega(a_{i}, z)\\
&&-\omega(x, [a_{i}, y])\omega(b_{i}, z)+\omega(x, [y, a_{i}])\omega(b_{i}, z)+\omega(x, a_{i})\omega([b_{i}, y], z)\\
&&-\omega(x, b_{i})\omega(y, [a_{i}, z])+\omega(x, b_{i})\omega(y, [z, a_{i}])+\omega(x, [b_{i}, z])\omega(y, a_{i})\Big)\\
&\stackrel{\eqref{eq:symplectic}}{=}&\sum_{i}\Big(\underbrace{\omega([y, a_{i}], x)\omega(b_{i}, z)}-\omega([y, x], a_{i})\omega(b_{i}, z)-\omega(a_{i}, y)\omega([b_{i}, x], z)\\
&&\underbrace{+\omega(x, b_{i})\omega([z, a_{i}], y)}-\omega(x, b_{i})\omega([z, y], a_{i})\underbrace{-\omega(x, [b_{i}, y])\omega(a_{i}, z)}\\
&&\underbrace{-\omega(x, [a_{i}, y])\omega(b_{i}, z)}+\underbrace{\omega(x, [y, a_{i}])\omega(b_{i}, z)}+\omega(x, a_{i})\omega([b_{i}, y], z)\\
&&-\omega(x, b_{i})\omega(y, [a_{i}, z])+\underbrace{\omega(x, b_{i})\omega(y, [z, a_{i}])}+\omega(x, [b_{i}, z])\omega(y, a_{i})\Big)\\
&=&\sum_{i}\Big(\underbrace{-\omega_{r}(a_{i}, y)\omega_{r}([b_{i}, x], z)+\omega_{r}(x, [b_{i}, z])\omega_{r}(y, a_{i})}\\
&&\underbrace{+\omega_{r}(x, a_{i})\omega_{r}([b_{i}, y], z)-\omega_{r}(x, b_{i})\omega_{r}(y, [a_{i}, z])}\\
&&-\omega_{r}([y, x], a_{i})\omega_{r}(b_{i}, z)-\omega_{r}(x, b_{i})\omega_{r}([z, y], a_{i})\Big)\\
&\stackrel{\eqref{eq:symplectic}}{=}&\sum_{i}\Big(\omega_{r}(a_{i}, y)\omega_{r}([z, x]-[x, z], b_{i})+\omega_{r}(a_{i}, x)\omega_{r}([z, y]-[y, z], b_{i})\\
&&+\omega_{r}([y, x], b_{i})\omega_{r}(a_{i}, z)-\omega_{r}(a_{i}, x)\omega_{r}([z, y], b_{i})\Big)\\
&\stackrel{\eqref{eq:symplectic}}{=}&\sum_{i}\Big(\omega_{r}(a_{i}, y)\omega_{r}([z, x]-[x, z], b_{i})-\omega_{r}(a_{i}, x)\omega_{r}([y, z], b_{i})+\omega_{r}([y, x], b_{i})\omega_{r}(a_{i}, z)\Big)\\
&=&\sum_{i}\Big(-\omega_{r}([z, x]-[x, z], y)+\omega_{r}([y, z], x)-\omega_{r}([y, x], z)\Big)\\
&\stackrel{\eqref{eq:symplectic}}{=}&0.
\end{eqnarray*}
Then by Theorem \ref{thm:ac}, the Nijenhuis operator $N$ on $(A, [,])$ is
\begin{eqnarray*}
&N(x)=\sum\limits_{i} \om_r(x, a_i)b_i=\sum\limits_{i}\langle {r^\#}^{-1}(x), a_i \rangle b_i=r^\#({r^\#}^{-1}(x))=x.&
\end{eqnarray*}}
That is to say, in this case, $N=\id_{A}$.
\end{ex}
\begin{ex}\label{ex:aff} Let $(A, [, ])$ be a 4-dimensional left-Alia algebra with basis $\{e_{1}, e_{2}, e_{3}, e_{4}\}$ and the product $[, ]$ given by the following table
\begin{center}
\begin{tabular}{r|rrrr}
$[,]$ & $e_{1}$ & $e_{2}$ & $e_{3}$ & $e_{4}$\\
\hline
$e_{1}$ & $0$ & $0$ & $0$ & $0$\\
$e_{2}$ & $0$ & $0$ & $0$ & $0$\\
$e_{3}$ & $e_{1}$ & $0$ & $0$ & $0$\\
$e_{4}$ & $e_{3}$ & $0$ & $0$ & $0$\\
\end{tabular}.
\end{center}
Set $r=e_{2}\o e_{3}-e_{3}\o e_{2}$. Then $(A, [, ], \D_{r})$ is a triangular left-Alia bialgebra, where $\D_{r}$ is given by:
\begin{center}
$\begin{cases}
\D_{r}(e_{1})=-e_{1}\o e_{2}-e_{2}\o e_{1}\\
\D_{r}(e_{2})=0\\
\D_{r}(e_{3})=0\\
\D_{r}(e_{4})=0\\
\end{cases}$.
\end{center}
Set
\begin{center}
\begin{tabular}{r|rrrr}
$\omega$ & $e_{1}$ & $e_{2}$ & $e_{3}$ & $e_{4}$\\
\hline
$e_{1}$ & $0$ & $0$ & $0$ & $0$\\
$e_{2}$ & $0$ & $0$ & $0$ & $\lambda$\\
$e_{3}$ & $0$ & $0$ & $0$ & $0$\\
$e_{4}$ & $0$ & $-\lambda$ & $0$ & $0$\\
\end{tabular}\ \ ($\lambda$ is a parameter).
\end{center}
Then $(A, [, ], \omega)$ is a symplectic left-Alia algebra, and $(A, [, ]_{\omega}, \D_{r})$ is a dual triangular left-Alia bialgebra. By Theorem \ref{thm:ac}, define 
\begin{center}
$\begin{cases}
N(e_{1})=0\\
N(e_{2})=0\\
N(e_{3})=0\\
N(e_{4})=-\lambda e_{3}\\
\end{cases}$.
\end{center}
 Then $((A, [, ]), N)$ is a Nijenhuis left-Alia algebra.
\end{ex}
In what follows, we construct Nijenhuis operators on a left-Alia coalgebra.
\begin{defi}\label{de:agg} Let $(A, \D)$ be a left-Alia coalgebra and $r \in A\o A$ be an antisymmetric element. Assume that
\begin{eqnarray}\label{eq:cosymplectic}
\sum_i (a_{i(1)}\o a_{i(2)}\o b_{i}-a_{i(2)}\o a_{i(1)}\o b_{i}-a_{i(2)}\o b_{i}\o a_{i(1)}+b_{i}\o a_{i(2)}\o a_{i(1)})=0.
\end{eqnarray}
Then, $(A, \D)$ is called a {\bf cosymplectic left-Alia coalgebra} denoted by $(A, \D, r)$.
\end{defi}

\begin{pro}\label{pro:ag} Let $(A, [, ])$ be a left-Alia algebra and $r \in A\o A$ be antisymmetric. If $(A, [, ], \D_{r})$ is a triangular left-Alia bialgebra, where $\D_{r}$ is defined by Eq.$\eqref{eq:comuti}$, then $(A, \D_{r}, r)$ is a cosymplectic left-Alia coalgebra.
\end{pro}
\begin{proof} Dual to Proposition $\ref{pro:ab}$.\end{proof} 

\begin{thm}\label{thm:aj} Let $(A, \D, r)$ be a cosymplectic left-Alia coalgebra and $\omega\in (A\o A)^{*}$. If $(A, [, ]_{\omega}, \D)$ is a dual triangular left-Alia bialgebra together with the multiplication $[, ]_{\omega}$ given in Eq.\eqref{eq:muti}, and further $(A, [, ]_{\omega}, \D_{r})$ is a triangular left-Alia bialgebra together with the comultiplication $\D_{r}$ given by Eq.\eqref{eq:comuti}, then $((A, \D), S)$ is a Nijenhuis left-Alia coalgebra, where $S:A\lr A$ is given by
\begin{equation*}
S(x):=\sum_{i}a_{i}\omega(b_{i}, x), \forall~ x\in A.
\end{equation*}
\end{thm}
\begin{proof} Dual to Theorem \ref{thm:ac}.\end{proof}
\section{Conclusion}
We end this paper with a question.
\begin{question} Theorem \ref{thm:ac} provides a method to construct Nijenhuis left-Alia algebras, and dually Nijenhuis left-Alia coalgebras can be obtained by Theorem \ref{thm:aj}. Under what conditions can a Nijenhuis left-Alia algebra $((A, [,]), N)$ given in Theorem \ref{thm:ac} and a Nijenhuis left-Alia coalgebra $((A, \D), S)$ given in Theorem \ref{thm:aj} produce a Nijenhuis left-Alia bialgebra $((A, [,], \D), N, S)$ given in Definition \ref{de:au}?
\end{question}

\section*{Acknowledgment} This work is supported by the National Natural Science Foundation of China (No. 12471033).

\smallskip
\noindent{\bf Declaration of interests.} The authors have no conflicts of interest to disclose.

\noindent{\bf Data availability.} Data sharing is not applicable as no new data were created or analyzed.


\begin{thebibliography}{99}

\bibitem{Ag00b} M. Aguiar, On the associative analog of Lie bialgebras. {\em J. Algebra} {\bf 244} (2001), 492-532.

\bibitem{Bai1} C. Bai, Double constructions of Frobenius algebras, Connes cocycles and their duality. {\em J. Noncommut. Geom.} {\bf 4} (2010), 475-530.

\bibitem{Dz} A. S. Dzhumadil'daev, Algebras with skew-symmetric identity of degree 3. {\em Sovrem. Mat. Prilozh.} {\bf 60} (2008), 13-31; translation in {\em J. Math. Sci. (N.Y.)} {\bf 161} (2009), 11-30.

\bibitem{KLWY} C. Kang, G. Liu, Z. Wang, S. Yu, Manin triples and bialgebras of left-Alia algebras associated to invariant theory. {\em Mathematics} {\bf 12} (2024), 408.

\bibitem{KLY} C. Kang, G. Liu, S. Yu, Yang-Baxter equations and relative Rota-Baxter operators for left-Alia algebras associated to invariant theory. {\em J. Nonlinear Math. Phys.} {\bf 31} (2024), 78.

\bibitem{MLo} T. Ma, L. Long, Nijenhuis operators and associative D-bialgebras. {\em J. Algebra} {\bf 639} (2024), 150-186.

\bibitem{Sw} M. E. Sweedler, {\em Hopf Algebras}. Benjamin, New York, 1969.

\bibitem{Zhe} V. N. Zhelyabin, Jordan bialgebras and their connection with Lie bialgebras (in Russian). {\em Algebra i Logika} {\bf 36} (1997), 3-25, 117; translation in {\em Algebra and Logic} {\bf 36} (1997), 1-15.

\end{thebibliography}
\end{document}